\documentclass[12 pt]{amsart}
\usepackage[utf8]{inputenc}
\usepackage{amsmath, amssymb, amsthm,verbatim, hyperref}
\usepackage[dvipsnames]{xcolor}
\usepackage{graphicx}
\usepackage{caption}
\usepackage{subcaption}
\usepackage{mathtools}
\usepackage{tikz, ytableau}
\usepackage{soul, array}
\usepackage[hmargin=1 in, vmargin = 1 in]{geometry}
\usepackage{cleveref}

\DeclareMathOperator{\ASM}{{\tt ASM}}

\DeclareMathOperator{\init}{{\tt in}}
\DeclareMathOperator{\supp}{{\tt Supp}}
\DeclareMathOperator{\del}{{\tt del}}
\DeclareMathOperator{\lk}{{\tt lk}}
\DeclareMathOperator{\codim}{{\tt codim}}
\DeclareMathOperator{\ess}{{\tt Ess}}
\DeclareMathOperator{\Dom}{{\tt Dom}}
\DeclareMathOperator{\rk}{rk}
\DeclareMathOperator{\asm}{\tt ASM}
\DeclareMathOperator{\Perm}{Perm}

\theoremstyle{plain}

\newtheorem{thm}{Theorem}[section]
\newtheorem{prop}[thm]{Proposition}
\newtheorem{cor}[thm]{Corollary}

\newtheorem{lemma}[thm]{Lemma}
\newtheorem{conjecture}[thm]{Conjecture}
\newtheorem{examplex}[thm]{Example}
\newenvironment{example}
  {\pushQED{\qed}\examplex}
  {\popQED\endexamplex}

\theoremstyle{definition}
\newtheorem{definition}[thm]{Definition}
\theoremstyle{remark}

\definecolor{darkblue}{rgb}{0.0,0,0.7}

\newcommand{\newword}[1]{\textcolor{darkblue}{\textbf{\emph{#1}}}}

\title{Some Algebraic Properties of ASM Varieties}
\author[Axelrod-Freed]{Ilani Axelrod-Freed}
\address{Department of Mathematics, Massachusetts Institute of Technology, Cambridge, MA}
\email{ilani\_af@mit.edu}
\author[Hao]{Hanson Hao}
\address{Department of Mathematics, University of California, Berkeley, Berkeley, CA}
\email{hhao@berkeley.edu}
\author[Kendall]{Matthew Kendall}
\address{Department of Mathematics, Stanford University, Stanford, CA}
\email{matthewkendall@stanford.edu}
\author[Klein]{Patricia Klein}
\address{Department of Mathematics, Texas A\&M University, College Station, TX}
\email{pjklein@tamu.edu}
\author[Luo]{Yuyuan Luo}
\address{Department of Mathematics, Princeton University, Princeton, NJ}
\email{luo.yuyuan@princeton.edu}

\thanks{All authors were partially supported by RTG grant NSF grant DMS-1745638.  PK was partially funded by NSF grant DMS-2246962 and by the Travel Support for Mathematicians gift MP-TSM-00002939 from the Simons Foundation. PK worked on this project while she was a member at the Institute for Advanced Study with support from the Bob Moses Fund. She also worked on this project while she was a long-term visitor at the Fields Institute.  She thanks both the IAS and Fields for their hospitality and support.}

\begin{document}

\begin{abstract}
    Fulton's matrix Schubert varieties are affine varieties that arise in the study of Schubert calculus in the complete flag variety.  Weigandt showed that arbitrary intersections of matrix Schubert varieties, now called ASM varieties, are indexed by alternating sign matrices (ASMs), objects with a long history in enumerative combinatorics. It is very difficult to assess Cohen--Macaulayness of ASM varieties or to compute their codimension, though these properties are well understood for matrix Schubert varieties due to work Fulton.  In this paper we study these properties of ASM varieties with a focus on the relationship between a pair of ASMs and their direct sum.  We also consider ASM pattern avoidance from an algebro-geometric perspective.
\end{abstract}

\maketitle

\tableofcontents

\section{Introduction}
Much of modern Schubert calculus centers on Schubert varieties in the complete flag variety.  These Schubert varieties are closely related to affine varieties called matrix Schubert varieties, introduced by Fulton \cite{Ful92}.  Specifically, Schubert polynomials, introduced by Lascoux and Sch{\"u}tzenberger \cite{LS82}, form an additive basis of the cohomology ring of the complete flag variety.  Expanding products of Schubert polynomials in this basis is arguably the major open problem in Schubert calculus today.  Knutson and Miller \cite{KM05} showed that Schubert polynomials are also the multidegrees  of matrix Schubert varieties.  Moreover, a Schubert variety and its associated matrix Schubert variety share a codimension \cite{Ful92}; the Cohen--Macaulay property of all Schubert varieties can be inferred from the Cohen--Macaulay property of all matrix Schubert varieties \cite{Ful92, KM05}; a Schubert variety is Gorenstein if and only if its associated matrix Schubert variety is \cite{WY06}.

In the course of studying matrix Schubert varieties, one is naturally led to study intersections of several of them, for example as they arise in Gr\"obner degenerations and in inductive computations \cite{KW21}.  Weigandt \cite{Wei17} showed that arbitrary intersections of matrix Schubert varieties are indexed by what are called alternating sign matrices (ASMs), a generalization of permutation matrices.  Permutation matrices index Schubert varieties and matrix Schubert varieties.  ASMs were not invented in order to perform this indexing.  Rather, they are combinatorial objects with their own rich history:  In 1983, Mills, Robbins, and Rumsey \cite{MRR83} gave a conjecture for a closed form for the number of $n \times n$ ASMs.
The original proof was given by Zeilberger \cite{Zei96}, and a second proof was given by
Kuperberg \cite{Kup96} using the six-vertex model of statistical mechanics.

Matrix Schubert varieties are Cohen--Macaulay and admit a codimension formula which can be read easily from their indexing permutation \cite{Ful92}.  These facts are regularly exploited in studying the Schubert and Grothendieck polynomials of their associated permutations and in recent computations of Castelnuovo--Mumford regularity (see, e.g., \cite{RRRSDW21, PSW24}).  ASM varieties do not in general enjoy either of these advantages.

The purpose of this paper is to study the codimension and potential Cohen--Macaulayness of ASM varieties.  Let $X_A$ denote the ASM variety associated to $A$.  We focus on the embeddings $\asm(n) \hookrightarrow \asm(n+1)$ determined by $A \mapsto 1 \oplus A$ and, more generally, on ASMs obtained as direct sums of other ASMs.  This embedding is natural in light of the role its restruction to $S_n$ plays in the theory of Stanley symmetric functions (\cite{Sta84, Mac91}).

Let $X_A$ denote the ASM variety determined by the ASM $A$.  We show the following:

\begin{thm}[\Cref{prop:codimOfDirectSum}]
\hspace{1cm}

    \begin{enumerate}
        \item $\codim(X_{A_1 \oplus A_2}) = \codim(X_{A_1}) + \codim(X_{A_2})$.
        \item $X_{A_1 \oplus A_2}$ is equidimensional if and only if $X_{A_1}$ and $X_{A_2}$ are both equidimensional.
    \end{enumerate}
\end{thm}

We conjecture that (2) holds if equidimensional is replaced with Cohen--Macaulay (\Cref{conj:CM}).  We prove one direction of the conjuncture and provide empirical evidence for the other (\Cref{prop:codimOfDirectSum}).

We give examples to show that the method used by Knutson and Miller \cite{KM05} to show that matrix Schubert varieties are Cohen--Macaulay, recovering a result of Fulton \cite{Ful92}, does not extend to the Cohen--Macaulay ASMs (\Cref{ex:non-KM-gvd} and \Cref{KM-vert-decomp-table}).  We also give data suggesting that the property of being Cohen--Macaulay becomes increasingly rare as the size of the ASM grows (\Cref{fig:CMData}).  These data motivate the study of ASM pattern avoidance from an algebro-geometric perspective.  We argue that the naive notion of pattern avoidance does not explain the eventual rarity of Cohen--Macaulay ASMs (\Cref{ex:CM-contains-non} and \Cref{containment-restrictions}) but that there is still some algebro-geometric control enforced by this notion (\Cref{prop:not-equidimensional}).  We hope that this will provide context for future studies of pattern avoidance of ASMs from an algebro-geometric perspective.

\section{Background}\label{sec:background}

Throughout this document, $\kappa$ will denote an arbitrary field.  For $k \leq n \in \mathbb{Z}_+$, let $[n] = \{1, 2, \ldots, n\}$ and $[k,n] = \{k, k+1, \ldots, n\}$. Let $S_n$ denote the symmetric group on $[n]$.  We identify $w \in S_n$ with the $n \times n$ permutation matrix with $1$s in positions $(i,w(i))$, $i \in [n]$, and $0$s elsewhere.

\begin{definition}\label{def: ASM}
An \newword{alternating sign matrix (ASM)} is square matrix with the following properties:
\begin{enumerate}
    \item Each entry is taken from the set $\{-1,0,1\}$.
    \item The entries in each row (resp. column) sum to 1.
    \item The nonzero entries in a row (resp. column) alternate between 1 and $-1$. 
\end{enumerate}
\end{definition}

Let $\asm(n)$ denote the set of $n \times n$ ASMs.  The ASMs whose entries all lie in $\{0,1\}$ are exactly the permutation matrices.  For $A \in \asm(n)$, note that the sum of entries along each row (resp. column) is $1$ and that the sum of all entries of $A$ is $n$.

\subsection{ASM varieties and ideals}
We now outline the process for associating a variety and an ideal to an alternating sign matrix.

Fix $A = (A_{a,b}) \in \asm(n)$.  Define a \newword{rank function} $\rk_A$ on $[n] \times [n]$ by
\[
\rk_A(i,j)=\sum_{a=1}^i\sum_{b=1}^j A_{a,b}.
\]
The \newword{rank matrix} of $A$ is the $n \times n$ matrix whose $(i,j)$ entry is $\rk_A(i,j)$. (Note that $\rk_A(i,j)$ is typically not the rank of the submatrix of $A$ consisting of the first $i$ rows and first $j$ columns, unless $A \in S_n$.)

THroughout this paper, we will let $Z = (z_{i,j})$ denote an $n \times n$ generic matrix, i.e., a matrix of distinct variables, and we will let $R$ denote the polynomial ring over $\kappa$ in the entries of $Z$.  For $i,j \in [n]$, let $Z_{[i],[j]}$ be the submatrix of $Z$ consisting of the first $i$ rows and $j$ columns.  For $k \in \mathbb{Z}_+$, let $I_k(Z_{[i],[j]})$ denote the ideal of $R$ generated by the $k \times k$ minors of $Z_{[i],[j]}$.

\begin{definition}
\label{def:ASMIdeal}
For $A \in \asm(n)$, let \[
    I_A = \sum_{(i,j) \in [n] \times [n]} I_{\rk_A(i,j)+1}(Z_{[i],[j]}).
\]  We call $I_A$ the \newword{ASM ideal} of $A$.  Let $X_A$ denote the variety of $I_A$, which we call the \newword{ASM variety} of $A$.
\end{definition}

Every ASM ideal $I_A$ is radical (\cite[Proposition 5.4]{Wei17}). See also \cite[Proposition 3.3]{Ful92} for the case $A \in S_n$.  When $A \in S_n$, $I_A$ is called a \newword{Schubert determinantal ideal} and $X_A$ a \newword{matrix Schubert variety}.  An ASM variety is irreducible if and only if it is a matrix Schubert variety \cite[Proposition 3.3]{Ful92}, \cite[Proposition 5.4]{Wei17}.  The codimension of $X_w$ is equal to the \newword{Coxeter length} of $w$, denoted $\ell(w)$ (\cite[Proposition 3.3]{Ful92}), also known as the inversion number of $w$.  See \cite{Ful92} for further information on matrix Schubert varieties, including their close connection to Schubert varieties in the complete flag variety.   

We call the generators described in Definition \ref{def:ASMIdeal} the \newword{natural generators} of $I_A$.  It is sometimes more convenient to work with a smaller set of generators called the Fulton generators, which we describe below, after giving some required auxiliary constructions.  Fulton generators of ASM ideals were introduced by Weigandt \cite{Wei17}, generalizing the description given by Fulton \cite{Ful92} for Schubert determinantal ideals.

For $A \in \asm(n)$, let \[
D(A) = \left\{(i,j) \mid \sum_{k = 1}^j A_{k,j} = 0 = \sum_{\ell = 1}^i A_{i,\ell} \right\}.
\] We call $D(A)$ the \newword{Rothe diagram} of $A$.  (Some authors refer to $D(A)$ as the inversion set of $A$.)  Note that the definition of ASM forces $\sum_{k = 1}^j A_{k,j}, \sum_{\ell = 1}^i A_{i,\ell} \in \{0,1\}$ for all $i,j \in [n]$.

\begin{example}
\label{ex:rotheDiagram}
We give an example of a method for visualizing the Rothe diagram of an ASM. If $A=\begin{pmatrix} 
    0 & 0 & 1 & 0\\
    1 & 0  & -1 & 1 \\
    0 & 1 & 0 & 0\\
    0 & 0 & 1 & 0
    \end{pmatrix}$, we may visualize $D(A)$ by drawing
    
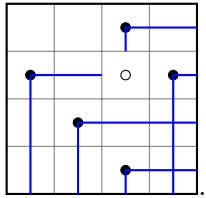
\[
   \begin{tikzpicture}[x=1.5em,y=1.5em]
\draw[step=1,gray, thin] (0,0) grid (4,4);
\draw[color=black, thick](0,0)rectangle(4,4);
\filldraw [black](0.5,2.5)circle(.1);
\filldraw [black](1.5,1.5)circle(.1);
\filldraw [black](2.5,0.5)circle(.1);
\draw [black](2.5,2.5)circle(.1);
\filldraw [black](2.5,3.5)circle(.1);
\filldraw [black](3.5,2.5)circle(.1);

\draw[thick, color=blue] (.5,0)--(.5,2.5)--(2,2.5);
\draw[thick, color=blue] (2.5,3)--(2.5,3.5)--(4,3.5);
\draw[thick, color=blue] (1.5,0)--(1.5,1.5)--(4,1.5);
\draw[thick, color=blue] (2.5,0)--(2.5,0.5)--(4,0.5);
\draw[thick, color=blue] (3.5,0)--(3.5,2.5)--(4,2.5);
\end{tikzpicture}.
\] 
In an $n \times n$ grid, we have placed a solid dot for every $1$ in $A$ and an open dot for every $-1$.  Lines emanate down and to the right out of each solid dot, stopping before entering the cell of an open dot.  Then $D(A) = \{(1,1), (1,2), (2,3)\}$ consists of the indices of cells without a line intersecting their interior.
\end{example}

Let $\ess(A) = \{(i,j) \in D(A) \mid (i,j+1), (i+1,j) \notin D(A)\}$.  We call $\ess(A)$ the \newword{essential set} of $A$ and elements of $\ess(A)$ the \newword{essential cells} of $A$.  

Returning to Example \ref{ex:rotheDiagram}, we have $\ess(A) = \{(1,2),(2,3)\}$.  We may visualize $\ess(A)$ as the maximally southeast elements of connected components of $D(A)$ (where two cells are considered adjacent if they share an edge).

It is a theorem due to Fulton \cite[Lemma 3.10]{Ful92} for $A \in S_n$ and generalized for all $A \in \asm(n)$ by Weigandt \cite[Lemma 5.9]{Wei17} that the ASM ideal $I_A$ may be generated solely by considering rank conditions at essential cells: \[
I_A = \sum_{(i,j) \in \ess(A)} I_{\rk_A(i,j)+1}(Z_{[i],[j]}).
\]  We call this set of generators of $I_A$ the \newword{Fulton generators}.  

Continuing with Example \ref{ex:rotheDiagram}, the rank matrix of $A$ is  $\rk_A=\begin{pmatrix} 
    0 & \boxed{0} & 1 & 1\\
    1 & 1  & \boxed{1} & 2 \\
    1 & 2 & 2 & 3\\
    1 & 2 & 3 & 4
    \end{pmatrix}$, where the ranks in essential cells appear in boxes.  The Fulton generators of $I_A$ are \[
\left\{z_{1,1} , z_{1,2}, \begin{vmatrix} z_{1,1} & z_{1,2} \\
z_{2,1} & z_{2,2} \end{vmatrix}, \begin{vmatrix} z_{1,1} & z_{1,3} \\
z_{2,1} & z_{2,3} \end{vmatrix}, \begin{vmatrix} z_{1,2} & z_{1,3} \\
z_{2,2} & z_{2,3} \end{vmatrix}\right\}.
\] The degree $1$ generators are determined by the essential cell $(1,2)$, and the degree $2$ are determined by the essential cell $(2,3)$.

    For $A \in \ASM(n)$, let $\Dom(A) = \{(i,j) \in [n] \times [n] \mid \rk_A(i,j) = 0\}$, and call $\Dom(A)$ the \newword{dominant} part of $A$.  Note that $\Dom(A) \subseteq D(A)$ and that $\Dom(A)$ consists exactly of the indices of the degree $1$ Fulton generators of $I_A$, i.e., the generators which are single variables.  In Example \ref{ex:rotheDiagram}, $\Dom(A) = \{(1,1),(1,2)\}$.

\section{Stability under embeddings, and limitations}
It is of central import in Schubert calculus that core constructions (e.g., the (double) Schubert polynomial, the defining equations of matrix Schubert variety) depending on the permutation $w \in S_n$ do not change if we instead view $w$ as an element of $S_{n+1}$ under the embedding $S_n \hookrightarrow S_{n+1}$ where each function is extended by a fixed point at $n+1$.  Constructions and invariants vary, though predictably, if instead we consider the embedding $S_n \hookrightarrow S_{n+1}$ where $w \in S_n$ is mapped to the function $w' \in S_{n+1}$ where $w'(1) = 1$ and $w'(i) = w(i-1)$ for $2 \leq i \leq n+1$.   For example, the assignment $(i,j) \mapsto (i+1,j+1)$ gives a bijection between $D(w)$ and $D(w')$; $\rk_w(i,j) = \rk_{w'}(i+1,j+1)-1$ for all $(i,j) \in D(w)$; $\codim(X_w) = \codim(X_{w'})$; and the Castelnuovo–Mumford regularity (an important algebraic invariant) of one may be inferred easily from the other via \cite{PSW24}.  This embedding plays an important role in the rich theory of Stanley symmetric functions, introduced by Stanley \cite{Sta84}, which Macdonald \cite{Mac91} showed to be the stable limit of Schubert polynomials.   

For $A \in \asm(n)$, let $1 \oplus A$ denote the direct sum of the $1 \times 1$ identity matrix with $A$ and $A \oplus 1$ the direct sum of $A$ with the $1 \times 1$ identity matrix, both of which are easily seen to be elements of $\asm(n+1)$.  As with permutations, we note that $I_A$ and $I_{A \oplus 1}$ have the same set of Fulton generators (in different polynomial rings).  Thus, $X_A$ and $X_{A \oplus 1}$ differ only by products with affine factors.  This is the extension of the first embedding of $S_n \hookrightarrow S_{n+1}$ discussed above.  

Motivated by the similarities between $I_w$ and $I_{1 \oplus w}$, $w \in S_n$, we discuss the relationship between the ASM ideals $I_A$ and $I_{1 \oplus A}$.  That is, we consider the natural extension of the second embedding of $S_n \hookrightarrow S_{n+1}$ discussed above.  Specifically, we give a codimension-preserving bijection between the components of $X_A$ and the components of $X_{1 \oplus A}$ (Proposition \ref{prop:permBij}), which can be viewed as an extension of the elementary fact that $\ell(w) = \ell(1 \oplus w)$ for $w \in S_n$.  

By contrast, we present data showing that a vertex decomposition of a simplicial complex naturally associated to $A$ does not necessarily give rise to a vertex decomposition of the corresponding simplicial complex of $1\oplus A$ (Figure \ref{KM-vert-decomp-table}).  This failure is interesting because vertex decompositions of these simplicial complexes have algebro-geometric interpretations via \cite{KR21}.  Specifically, had that implication held, it would have allowed one to lift certain proofs that certain $X_A$ belong to the Gorenstein linkage class of a complete intersection to proofs that the corresponding $X_{1 \oplus A}$ do.  For the relevant Gorenstein linkage background, we direct the reader to \cite{NR08, KR21}.

This contrast motivates our interest in the relationship between Cohen--Macaulayness of $A$ and $1 \oplus A$.   Every variety in the Gorenstein linkage class of a complete intersection is Cohen--Macaulay, and every Cohen--Macaulay variety is equidimensional.  In this sense, the Cohen--Macaulay property is a ``middle ground" between the two previously-discussed conditions.  The equidimensionality result (Proposition \ref{prop:permBij}) together with computer data lead us to conjecture that $A$ is Cohen--Macaulay if and only if $1 \oplus A$ is (Conjecture \ref{conj:CM}).  

\subsection{Equidimensionality}
If $A,B$ are $n \times n$ ASMs, define $A \ge B$ if $\rk_A(i,j) \le \rk_B(i,j)$ for all $1 \le i,j \le n$.
Restricted to permutation matrices, this is the {\it (strong) Bruhat order} on $S_n$.
Let
\[
    \Perm(A) = \{w \in S_n : \text{$w \ge A$ and if $w \ge v \ge A$ for some $v \in S_n$, then $v=w$}\}.
\]

\begin{prop}\cite[Proposition 5.4]{Wei17}\label{prop:perm-set-decomposition} $I_A$ has the irredundant prime decomposition
\[
    I_A = \bigcap_{w \in \Perm(A)} I_w.
\]
The codimension of $X_A$ is $\min \{\ell(w) \mid w \in \Perm(A)\}$.
\end{prop}

\begin{definition}\label{def: equidim}
Call $A$ \newword{equidimensional} if all elements of $\Perm(A)$ have the same Coxeter length, or, equivalently, if $X_A$ is equidimensional.
\end{definition}

\begin{prop}\label{prop:permBij}
Let $A \in \ASM(n)$ and $\Perm(A) =\{w_1,\ldots,w_k\}$.  For $w \in S_n$, the assignment $w \mapsto 1 \oplus w$ gives a bijection between $\Perm(A)$ and $\Perm(1 \oplus A)$.  In particular, $
\codim(A)=\codim(1 \oplus A)$, and $A$ is equidimensional if and only if $1 \oplus A$ is equidimensional.
\end{prop}
\begin{proof}
If $A \in S_n$, then the claim is trivial, and so we assume $A \notin S_n$.

Suppose $w \in \Perm(A)$.  We will first show that $1 \oplus w \in \Perm(1 \oplus A)$.  From the definition of $1 \oplus A$, we compute \[\rk_{1 \oplus A}(i,j) = 
\begin{cases}
1 & i=1 \mbox{ or } j=1\\
\rk_A(i-1,j-1)+1 & \mbox{otherwise}.
\end{cases}
\]

Thus
\begin{equation}\label{eq:perm_1}
    \text{$w > A$ if and only if $1 \oplus w > {1 \oplus A}$.}
\end{equation}
Suppose there exists a $v$ such that $1 \oplus w \ge v > 1 \oplus A$.  From the equations
\[
    \rk_{1 \oplus w}(1,j) \le \rk_{v}(1,j) \le \rk_{1 \oplus A}(1,j)
\] and \[
 \rk_{1 \oplus w}(1,j) = 1 = \rk_{1 \oplus A}(1,j)
\]
for all $j \in [n+1]$, we see that the first row of $v$ is the same as the first row of $1 \oplus A$, which is the same as the first row of $1 \oplus w$ - specifically, the $1^{st}$ unit row vector.  Similarly, the first column of all three matrices $1 \oplus w$, $v$, and $1 \oplus A$ must be the $1^{st}$ unit column vector.  

Thus $v = 1 \oplus v'$ for some $v' \in S_n$.  By (\ref{eq:perm_1}), $w \geq v'>A$.  By definition of $\Perm(A)$, the inequality $w \geq v' > A$ implies $w=v'$, and so $1 \oplus w = 1 \oplus v'=v$.  Hence, $1 \oplus w \in \Perm(1 \oplus A)$.

In the other direction, fix some $u \in \Perm(1 \oplus A)$.  We claim that $u = 1 \oplus u'$ for some $u' \in S_n$.  It suffices to show that $u(1) = 1$.

Suppose for contradiction that $u(1) = j$ for some $j>1$ and that $i$ is the least index so that $u(i)<j$.  Let $\nu = ut_{1,i}$, where $t_{1,i}$ is the transposition $(1i)$.  That is, in matrix form, $\nu$ is the matrix obtained from $u$ by exchanging rows $1$ and $i$. 

We claim that $1 \oplus A \leq \nu <u$, which will contradict the assumption $u \in \Perm(1 \oplus A)$.  Note that $\rk_\nu(\alpha,\beta) = \rk_u(\alpha,\beta)$ unless $\alpha \in [i-1]$ and $\beta \in [u(i), j-1]$, in which case $\rk_u(\alpha,\beta) = \rk_\nu(\alpha,\beta)-1$ (see \Cref{ex:permBij} for a visualization). Thus $\nu<u$.  Because $i$ was chosen minimally, $\rk_u(\alpha,\beta) = 0$ for all $\alpha \in [i-1]$ and $\beta \in [u(i), j-1]$, and so $\rk_\nu(\alpha,\beta) = 1$ for all such $(\alpha, \beta)$.  But $\rk_{1 \oplus A}(\alpha, \beta) \ge 1$ for all $\alpha,\beta$.  It follows from $\rk_\nu(\alpha,\beta) = \rk_u(\alpha,\beta) \leq \rk_{1 \oplus A}(\alpha, \beta)$ unless $\alpha \in [i-1]$ and $\beta \in [u(i), j-1]$ together with $\rk_\nu(\alpha,\beta) = 1 \leq \rk_{1 \oplus A}(\alpha,\beta)$ for $\alpha \in [i-1]$ and $\beta \in [u(i), j-1]$ that $1 \oplus A \leq \nu$, completing the claim.

Hence our arbitrary $u \in \Perm(1 \oplus A)$ has the form $u = 1 \oplus u'$ for some $u' \in S_n$. 
 It remains to show that $u' \in \Perm(A)$. 
 By (\ref{eq:perm_1}), $u' >A$.  Hence, if $u' \notin \Perm(A)$, then there exists some $\tilde{u} \in S_n$ satisfying $A < \tilde{u}<u'$.  But then, by (\ref{eq:perm_1}), $1 \oplus A < 1 \oplus \tilde{u}<u$, contradicting the assumption $u \in \Perm(1 \oplus A)$.
\end{proof}

\begin{example}
    In order to help internalize the argument of \Cref{prop:permBij}, we will give a visualization of the region in which $\rk_u$ and $\rk_\nu$ differ for some $\nu = ut_{1,i}$.  Consider $u = 45213$, in which case $j =4$, $i = 3$, $u(i) = 2$, and $\nu = 25413$.  The region $[3-1] \times [2,4-1]$, in which $\rk_u$ and $\rk_\nu$ differ, is shaded in yellow.  

        \[
   \begin{tikzpicture}[x=1.5em,y=1.5em]
  \draw[color=black, thick](0,0)rectangle(5,5);
   \fill [yellow] (1,3) rectangle (3,5);
\draw[step=1,gray, thin] (0,0) grid (5,5);
\node at (-1,2.5){$u:$};

\filldraw [black](0.5,1.5)circle(.1);
\filldraw [black](1.5,2.5)circle(.1);
\filldraw [black](2.5,0.5)circle(.1);
\filldraw [black](3.5,4.5)circle(.1);
\filldraw [black](4.5,3.5)circle(.1);

\draw[thick, color=blue] (.5,0)--(.5,1.5)--(5,1.5);
\draw[thick, color=blue] (2.5,0)--(2.5,.5)--(5,.5);
\draw[thick, color=blue] (1.5,0)--(1.5,2.5)--(5,2.5);
\draw[thick, color=blue] (3.5,0)--(3.5,4.5)--(5,4.5);
\draw[thick, color=blue] (4.5,0)--(4.5,3.5)--(5,3.5);
\end{tikzpicture} \hspace{2cm}       \begin{tikzpicture}[x=1.5em,y=1.5em]
  \draw[color=black, thick](0,0)rectangle(5,5);
    \fill [yellow] (1,3) rectangle (3,5);
\draw[step=1,gray, thin] (0,0) grid (5,5);
\node at (-1,2.5){$\nu:$};

\filldraw [black](0.5,1.5)circle(.1);
\filldraw [black](1.5,4.5)circle(.1);
\filldraw [black](2.5,0.5)circle(.1);
\filldraw [black](3.5,2.5)circle(.1);
\filldraw [black](4.5,3.5)circle(.1);

\draw[thick, color=blue] (.5,0)--(.5,1.5)--(5,1.5);
\draw[thick, color=blue] (2.5,0)--(2.5,.5)--(5,.5);
\draw[thick, color=blue] (1.5,0)--(1.5,4.5)--(5,4.5);
\draw[thick, color=blue] (3.5,0)--(3.5,2.5)--(5,2.5);
\draw[thick, color=blue] (4.5,0)--(4.5,3.5)--(5,3.5);
\end{tikzpicture}   \qedhere
\] 
\end{example}

\begin{example}\label{ex:permBij}
 \Cref{prop:permBij} involves inserting into an ASM a row and column whose entries are $0$ except where they intersect, where the value is $1$. In the case of \Cref{prop:permBij}, it is specifically the first row and first column.  It is worth noting that equidimensionality is \emph{not} preserved by an arbitrary insertion of this type.  For example, consider $A = \begin{pmatrix}
0 & 1 & 0 \\
1 & -1 & 1 \\
0 & 1 &  0
\end{pmatrix}$ and $B = \begin{pmatrix}
0 & 1 & 0 & 0 \\
0 & 0 & 1 & 0 \\
1 & -1 & 0 & 1 \\
0 & 1 & 0 & 0
\end{pmatrix}$, in which case $B$ may be obtained from $A$ by inserting $(0,1,0,0)^T$ to become column $3$ and $(0,0,1,0)$ to become row $2$.

Then $I_A = (z_{1,1}, z_{1,2}z_{2,1}) = I_{312} \cap I_{231}$ note only defines an equidimensional ASM variety but even complete intersection while $I_B = (z_{1,1}, z_{2,1}, z_{1,2}z_{3,1}, z_{2,2}z_{3,1}) = I_{3412} \cap I_{2341}$ has one component of codimension $4$ and one of codimension $3$.  
\end{example}

\subsection{Vertex decomposition}

We now review some basic definitions from simplicial complex theory.  A \newword{simplicial complex} $\Delta$ on $[n]$ is a set of subsets of $[n]$ such that, if $\sigma \in \Delta$ and $\tau \subseteq \sigma$, then $\tau \in \Delta$.  We call the elements of $\Delta$ \newword{faces} and the maximal faces (by inclusion) \newword{facets}.  The \newword{dimension} of a face $\sigma \in \Delta$ is $|\sigma|-1$, and the dimension of $\Delta$ is the maximum among dimensions of its faces.  Faces of dimension $0$ are called $\newword{vertices}$.  We call $\Delta$ \newword{pure} if all of its facets have the same dimension.

Now we discuss how these ideas from simplicial complex theory relate to the objects we want to study. The \newword{Stanley--Reisner} correspondence is a bijection between simplicial complexes on $[n]$ and squarefree monomial ideals. For a subset $\sigma$ of $[n]$, let $\mathbf{x}_\sigma = \prod_{i \in \sigma} x_i$. The Stanley--Reisner ideal of the simplicial complex $\Delta$ is 
\[
    I_{\Delta} = (\mathbf{x}_{\sigma} \mid \sigma \notin \Delta).
\]

The complements of the facets of $\Delta$ correspond to the associated primes of $I_\Delta$.  Hence, $\Delta$ is pure if and only if the variety defined by $I_\Delta$ is equidimensional.  We call $\Delta$ \newword{Cohen--Macaulay} whenever $I_\Delta$ defines a Cohen--Macaulay quotient ring.

Given a simplicial complex $\Delta$, we define the \newword{link} of $\Delta$ at a face $\sigma$ of $\Delta$ by \[
\lk_{\sigma}(\Delta) = \{\tau \in \Delta \mid \tau \cap \sigma = \emptyset, \tau \cup \sigma \in \Delta\}
\] and the \newword{deletion} of $\Delta$ at a face $\sigma$ of $\Delta$ by \[
\del_\sigma(\Delta) = \{\tau \in \Delta \mid \tau \cap \sigma = \emptyset\}.
\]  Notice that $\lk_\sigma(\Delta)$ is a subcomplex of $\del_{\sigma}(\Delta)$.  When $\sigma = \{v\}$ is a vertex of $\Delta$, we will often write $\lk_{v}(\Delta)$ and $\del_{v}(\Delta)$ for $\lk_{\sigma}(\Delta)$ and $\del_{\sigma}(\Delta)$, respectively.

A simplicial complex $\Delta$ is \newword{vertex decomposable} if $\Delta$ is pure and if either
    \begin{enumerate}
        \item $\Delta = \{\emptyset\}$, or if
        \item for some vertex $v$ in $\Delta$, both $\del_{v}(\Delta)$ and $\lk_{v}(\Delta)$ are vertex decomposable.
    \end{enumerate}

Vertex decomposable simplicial complexes were introduced by Provan and Billera \cite{PB80}, who showed that vertex decomposable simplicial complexes are shellable, known more classically to be a virtue of a simplicial complex.  Reisner \cite{Rei76} gave a topological criterion characterizing Cohen--Macaulayness of $\Delta$ from which it follows that every shellable simplicial complex is Cohen--Macaulay, a desirable algebraic property we discuss further in \Cref{sec:CM}.

We will use Stanley--Reisner theory to study initial ideals of ASM ideals. We refer the reader to \cite[Chapter 1]{miller-sturmfels} for more information on Stanley--Reisner theory and to \cite[Chapter 15]{Eis95} for general background on term orders and Gr\"obner bases.

Recall that $Z$ is a fixed $n \times n$ generic matrix.  We call a term order \newword{antidiagonal} if the lead term of the determinant of any submatrix $Y$ of $Z$ is the product of the entries along the antidiagonal of $Y$.  For each $A \in \asm(n)$, the Fulton generators of $I_A$ form a Gr\"obner basis with respect to any antidiagonal term order (\cite{KM05, Knu, Wei17}).  (For $w \in S_n$, Gao and Yong \cite{GY24} recently determined a minimal generating set of $I_w$ which is already a Gr\"obner basis under any antidiagonal term order.)  In particular, for any $A \in \asm(n)$, there is only one initial ideal arising from an antidiagonal term order.  We call that initial ideal the {\newword{antidiagonal initial ideal} of $I_A$, which we denote $\init I_A$.

Knutson and Miller \cite{KM05} showed that, for $w \in S_n$, $\init I_w$ is not only Cohen--Macaulay but, moreover, the Stanley--Reisner ideal of a vertex decomposable simplicial complex.  In fact, for a fixed $n$, there is an ordering of the vertices (labeling the vertices by their associated variable)
\[
    z_{1,n} > z_{1,n-1} > \cdots > z_{1,1} > z_{2,n} > z_{2,n-1} > \cdots > z_{n,1}.
\]
by which the Stanley--Reisner complex of each $\init I_w$ is vertex decomposable, independent of $w$.  We may hope that that same term order would show that every Stanley--Reisner ideal of every Cohen--Macaulay $\init I_A$ is vertex decomposable. It turns out that that is false.

\begin{example}\label{ex:non-KM-gvd}
For example, let
\[
A = \begin{pmatrix} 0&1&0&0\\ 0&0&0&1\\ 1&-1&1&0\\ 0&1&0&0 \end{pmatrix}, \;\; \init I_A = (z_{11},z_{21},z_{12}z_{31},z_{31}z_{22},z_{22}z_{13}).
\]

The ideal of the deletion at that vertex corresponding to $z_{13}$ is
\[
 (z_{11},z_{21},z_{12}z_{31},z_{31}z_{22}) = (z_{11}, z_{21}, z_{12}, z_{22}) \cap (z_{11}, z_{21}, z_{31}),
\]
which has associated primes of different heights, which shows that the deletion of the Stanley--Reisner complex of $\init I_A$ at the vertex corresponding to $z_{13}$ is not pure.  Hence, the Stanley--Reisner comples of $\init I_A$ is not vertex decomposable under the order use by Knutson and Miller.
\end{example}

If the Stanley--Reisner complex of $\init I_A$ is vertex decomposable by the order dictated by \cite{KM05}, we will call $A$ \newword{Knutson--Miller vertex decomposable}, a definition which is local to this paper.  We use the abbreviation CM for Cohen--Macaulay.
The middle column of \Cref{KM-vert-decomp-table} records the number of elements of $\asm(n)$ that are obtained as $1 \oplus A$ from some $A \in \asm(n-1)$ (which is equivalent to satisfying $A_{1,1} = 1$) and that are not Knutson--Miller vertex decomposable.  Note, for example, that there are two such elements of $\asm(5)$ but only one element of $\asm(4)$ that is not Knutson--Miller vertex decomposable.  Hence, there exists at least one element of $\asm(5)$ of the form $1 \oplus A$ that is not Knutson--Miller vertex decomposable though $A$ is.  This example shows that the operation $1 \oplus -$ does not preserve Knutson--Miller vertex decomposability. This failure also appears to becomes more frequent as $n$ grows.  It is in part due to the frequency of this failure for small $n$ that we are meaningfully encouraged by the preservation of Cohen--Macaulayness under $1 \oplus -$ applied to all $A \in \asm(n)$ for $n \leq 6$, which we discuss further in \Cref{conj:CM}.

\begin{figure}[h]
\begin{center}
\begin{tabular}{ | m{.5em} | m{4.9cm}| m{4.9cm} | m{4.9cm} |} 
  \hline
  n & Number of CM ASMs that are not Knutson--Miller vertex decomposable & Number of CM ASMs satisfying $A_{1,1}=1$ that are not Knutson--Miller vertex decomposable & Percentage of CM ASMs that are not Knutson--Miller vertex decomposable\\ 
  \hline
  4 & 1 & 0 & 2.6\% \\
 \hline
 5 & 35 & 2 & 10.7\%\\
 \hline
 6 & 1033 & 60 & 25.6\%\\
  \hline
  7 & 31,596 & 1538 & 45.0\%\\
  \hline
\end{tabular}
\caption{Count of Cohen--Macaulay elements of $\asm(n)$ that are and are not vertex decomposable by the Knutson--Miller ordering of the variables.}
\label{KM-vert-decomp-table}
\end{center}
\end{figure}

\subsection{Cohen--Macaulayness}\label{sec:CM}

Cohen--Macaulayness of rings and varieties is a central consideration in commutative algebra and algebraic geometry, especially in the context of intersection theory. As Hochster said, ``Life is really worth living in a Noetherian ring $R$ when all the local rings
have the property that every [system of parameters] is an $R$-sequence. Such a ring is called \emph{Cohen-Macaulay}" \cite[P.887]{Hoc78}. For background on Cohen--Macaulay rings, we direct the reader to \cite{BH93} and, especially for their import in combinatorial settings, to \cite{Hoc16}.

\begin{definition} We will say that \newword{the ASM $A$ is Cohen--Macaulay} if the ring $R/I_A$ is Cohen--Macaulay. 
\end{definition}

In light of recent work of Conca and Varbaro \cite{CV20}, we will be able to study the Cohen--Macaulayness of each $R/I_A$ via the Cohen--Macaulayness of a suitable Gr\"obner degeneration.  

\begin{prop}[\cite{Wei17, Knu, KW21}]\label{prop:initialIdealSplit}  If $w_1, \ldots, w_r, u_1, \ldots, u_k \in S_n$ and $A \in \asm(n)$ are such that \[
I_{u_1}+\ldots+I_{u_k} = I_A = I_{w_1} \cap \cdots \cap I_{w_r},
\] then 
\[
\init I_{u_1}+\ldots+\init I_{u_k} = \init I_A = \init I_{w_1} \cap \cdots \cap \init I_{w_r}.
\]
\end{prop}

\begin{prop}[\cite{KM05, BergeronBilley}]\label{prop:components-give-reduced-words}
    Fix $w \in S_n$.  If $P$ is an associated prime of $\init I_w$, then the generators of $P$ determine a reduced word for $w$.  Hence, $P$ is not an associated prime of any $\init I_u$ with $w \neq u \in S_n$.
\end{prop}

\begin{prop}
    \label{thm: initIdealCM}
Fix $A \in \asm(n)$.  Then $ R/I_A$ is Cohen--Macaulay if and only if $R/\init I_A$ is Cohen--Macaulay.
\end{prop}
\begin{proof}
The ideal $I_A$ is homogeneous.  Its initial ideal $\init I_A$ is radical by \Cref{prop:initialIdealSplit} because each $\init I_w$ is radical for $w \in S_n$ \cite[Theorem B]{KM05}. The main theorem of \cite{CV20} states that a homogeneous ideal with a radical initial ideal defines a Cohen--Macaulay quotient if and only if the radical initial ideal does.
\end{proof}

Our primary interest in what follows will be to infer Cohen--Macaulayness of $A$ from Cohen--Macaulayness of $1 \oplus A$.  Because the argument is essentially the same, we will prove something slightly stronger using the construction below.  In the notation that follows, $1 \oplus A = \widetilde{A}(1,1)$.

Fix $A \in \asm(n)$ and $i, j \in [n]$. We will build an element of $\asm(n+1)$, which we denote $\widetilde{A} = \widetilde{A}(i,j)$, by the following rules: 
\[
\widetilde{A}_{a,b}=\begin{cases}
			A_{a,b}, & \text{if $a<i$ and $b<j$}\\
            A_{a,b-1}, & \text{if $a<i$ and $b>j$}\\
            A_{a-1,b}, & \text{if $a>i$ and $b<j$}\\
            A_{a-1,b-1}, & \text{if $a>i$ and $b>j$}\\
            1, & \text{if $(a,b)=(i,j)$}\\
            0, & \text{if $a=i$ xor $b=j$}.
		 \end{cases}
\]

\begin{example}\label{ex:widetilde{A}}
    Returning to \Cref{ex:permBij}, consider $A = \begin{pmatrix}
        0 & 1 & 0\\
        1 & -1 & 1\\
        0 & 1 & 0
    \end{pmatrix}$ and then $\widetilde{A}  = \widetilde{A}(2,3)= \begin{pmatrix}
        0 & 1 & {\color{blue}0} & 0\\
        {\color{blue}0} & {\color{blue}0} & {\color{blue}1} & {\color{blue}0}\\
         1 & -1 & {\color{blue}0} & 1\\
        0 & 1 & {\color{blue}0} & 0
    \end{pmatrix}$ (which in \Cref{ex:permBij} was named $B$).  The entries of $\widetilde{A}$ inherited from $A$ appear in black, and the new entries appear in blue.  
\end{example}

The reader may verify that the Rothe diagram of $A$ embeds into the Rothe diagram of $\widetilde{A}$.  This close relationship may cause one to hope that the geometry of $X_A$ and $X_{\widetilde{A}}$ would be very closely related.  However, as we discussed in \Cref{ex:permBij}, $X_A$ may be Cohen--Macaulay while $X_{\widetilde{A}}$ is not even unmixed. We now investigate special cases when $X_A$ and $X_{\widetilde{A}}$ are closely related.

As before, for $\tau \subseteq [n]$, let $\mathbf{x}_\tau = \prod_{i \in \tau} x_i$.  Recall that, for an ideal $I$ and element $f$ of $R$, we define the colon ideal $(I:f) = (r \in R \mid rf \in I)$.

\begin{lemma}\label{lem:link-colon}
   Let $\Delta$ be a simplicial complex on $[n]$. If $\sigma$ is a face of $\Delta$, then \[
   I_{\lk_\sigma(\Delta)} = (I_{\Delta}: \mathbf{x}_\sigma)+(z_{i,j} \mid (i,j) \in \sigma).
   \]
\end{lemma}
\begin{proof}
    Fix a a squarefree monomial $\mathbf{x}_\tau$ corresponding to the subset $\tau$ of $[n]$.  If $\tau \cap \sigma \neq \emptyset$, then $\mathbf{x}_\tau \in (z_{i,j} \mid (i,j) \in \sigma)$.  Also, by the definition of link, $\tau \notin \lk_\sigma(\Delta)$, and so $\mathbf{x}_\tau \in I_{\lk_\sigma(\Delta)}$.  

    Now assume $\tau \cap \sigma = \emptyset$.  Then \[
\mathbf{x}_\tau \in I_{\lk_{\sigma(\Delta)}} \iff \tau \notin \lk_\sigma(\Delta) \iff \tau \cup \sigma \notin \Delta \iff \mathbf{x}_\tau \mathbf{x}_\sigma= \mathbf{x}_{\tau \cup \sigma} \in I_\Delta \iff \mathbf{x}_\tau \in I_\Delta:\mathbf{x}_\sigma. \qedhere
    \]
\end{proof}

For $A \in \asm(n)$, we write $\Delta(A)$ for the Stanley--Reisner complex of $\init I_A$.

\begin{thm}\label{topPermRow}
Let $A \in \asm(n)$, fix $j \in [n+1]$, and set $\widetilde{A} = \widetilde{A}(1,j)$.
If $\widetilde{A}$ is Cohen--Macaulay then $A$ is Cohen--Macaulay.
\end{thm}
\begin{proof}
    Recall that the antidiagonal initial ideals of ASM ideals are radical. Recall also that, if a homogeneous ideal $I$ possesses a radical initial ideal $J$, then $I$ is a Cohen--Macaulay ideal if and only if $J$ is a Cohen--Macaulay ideal \cite{CV20}.  Hence, it suffices to show that, if $\init I_{\widetilde{A}}$ is a Cohen--Macaulay ideal, then $\init I_A$ is a Cohen--Macaulay ideal.  We will show that, after a relabeling of vertices and deconing, $\Delta(A)$ is obtained as a link of $\Delta(\widetilde{A})$, from which the result will follow because links of Cohen--Macaulay simplicial complexes are Cohen--Macaulay \cite{Rei76}.

    For convenience, we will choose a non-standard indexing on the generic matrix from which equations for $I_{\widetilde{A}}$ are defined.  Let $Z' = (z_{a,b})$ be an $n \times n$ generic matrix with $0 \leq a \leq n$ and $1 \leq b \leq n+1$.  Correspondingly, index the rows of $\widetilde{A}$ with the set $\{0,1,\ldots, n\}$ and the columns with $\{1,\ldots,n+1\}$, and consider $\Delta(\widetilde{A})$ as a complex on $\{0,\ldots,n\} \times [n+1]$ with vertex $(a,b)$ corresponding to variable $z_{a,b}$ in the usual way.  We retain the usual indexing for $A$ and use $Z = (z_{a,b})$ with $a,b \in [n]$ to define $I_A$.

    \begin{example}
If $A = \begin{pmatrix}
        \boxed{0} & 1 & 0\\
        1 & \boxed{-1} & 1\\
        0 & 1 & 0
    \end{pmatrix}$ and $j=2$, then $\widetilde{A} = \begin{pmatrix}
          {\color{blue}0} & {\color{blue}1} & {\color{blue}0} & {\color{blue}0}\\
        \boxed{0} & {\color{blue}0} & 1 &  0\\
         1 & {\color{blue}0} & \boxed{-1}  & 1\\
        0 & {\color{blue}0} & 1  & 0
    \end{pmatrix}$.  With the prescribed indexing, $\ess(A) = \{(1,1),(2,2)\}$ and $\ess(\widetilde{A}) = \{(1,1),(2,3)\}$.  Note also that $\rk_{\widetilde{A}}(1,1) =\rk_A(1,1)$ (because column $1$ is to the left of column $j=2$) and that $\rk_{\widetilde{A}}(2,3) = \rk_A(2,2)+1$ (because column $3$ is to the right of column $j=2$).  Then \[
    I_A = \left(z_{1,1}, \begin{vmatrix} 
z_{1,1} & z_{1,2}\\
z_{2,1} & z_{2,2}
    \end{vmatrix}\right) \mbox{, and } I_{\widetilde{A}} = \left(z_{0,1},z_{1,1}, \begin{vmatrix} 
z_{0,1} & z_{0,2} & z_{0,3}\\
z_{1,1} & z_{1,2} & z_{1,3}\\
z_{2,1} & z_{2,2} & z_{2,3}
    \end{vmatrix}\right), 
    \] where the matrix entries corresponding to essential cells are boxed.
    \end{example}

    We claim that, with this choice of indexing, \begin{align}\label{CM-link-main-claim}
    \init I_A+(z_{0,1}, \ldots,z_{0,j-1}, z_{0,j+1}, \ldots,  z_{0,n+1}) &= (\init I_{\widetilde{A}}:z_{0,j+1}\cdots z_{0,n+1})+(z_{0,j+1}, \ldots, z_{0,n+1}) \\
    &= I_{\lk_{\{(0,j+1), \ldots, (0,n+1)\}}\Delta(\widetilde{A})}.\notag \end{align}  
    
    Because $I_A$ has a generating set that does not involve the variables $z_{0,1}, \ldots, z_{0,n+1}$, $\init I_A$ is a Cohen--Macaulay ideal if and only if $\init I_A+(z_{0,1}, \ldots z_{0,j-1}, z_{0,j+1}, \ldots , z_{0,n+1})$ is.  Similarly, $\init I_{\widetilde{A}}:z_{0,j+1}\cdots z_{0,n+1}$ has a generating set that does not involve the variables $z_{0,j+1}, \ldots, z_{0,n+1}$, and so $(\init I_{\widetilde{A}}:z_{0,j+1}\cdots z_{0,n+1})+(z_{0,j+1}, \ldots, z_{0,n+1})$ is a Cohen--Macaulay ideal because $\init I_{\widetilde{A}}:z_{0,j+1}\cdots z_{0,n+1}$ is.  Hence, the desired result will follow from establishing \Cref{CM-link-main-claim}.

    The latter equality of \Cref{CM-link-main-claim} follows from \Cref{lem:link-colon}. 
 In order to establish the first equality of \Cref{CM-link-main-claim}, we will first show \[
 \init I_A+(z_{0,1}, \ldots, z_{0,j-1}, z_{0,j+1}, \ldots, z_{0,n+1}) \subseteq (\init I_{\widetilde{A}}:z_{0,j+1}\cdots z_{0,n+1})+(z_{0,j+1}, \ldots, z_{0,n+1}).
 \] We note that the containments $(z_{0,1}, \ldots, z_{0,j-1}) \subseteq \init I_{\widetilde{A}} \subseteq \init I_{\widetilde{A}}:z_{0,j+1}\cdots z_{0,n+1}$ are immediate from the definitions of rank function, ASM ideal, and colon ideal.  Fix $\mu \in \init I_A$.  We may assume that $\mu$ is the leading term of some Fulton generator $f$ of $I_A$ determined by essential cell $(a,b)$ of $A$.  
    
    If $b<j$, then $(a,b)$ is also an essential cell of $\widetilde{A}$ and $\rk_A(a,b) = \rk_{\widetilde{A}}(a,b)$.  Hence, $f$ is a Fulton generator of $I_{\widetilde{A}}$, and so $\mu \in \init I_{\widetilde{A}} \subseteq \init I_{\widetilde{A}}:z_{0,j+1}\cdots z_{0,n+1}$.  
    
    If $b \geq j$, then $(a,b+1)$ is an essential cell of $\widetilde{A}$ and $\rk_A(a,b) = \rk_{\widetilde{A}}(a,b+1)-1$.  Let $M$ be the $(\rk_A(a,b)+1) \times (\rk_A(a,b)+1)$ submatrix of $Z$ so that $f = \det(M)$.  Form a submatrix $M'$ of $Z'$ by augmenting the set of rows of $M$ by $0$ and the set of columns of $M$ by $b+1$.  Then $M'$ is a $(\rk_{\widetilde{A}}(a,b+1)+1) \times \rk_{\widetilde{A}}(a,b+1)+1)$ submatrix of $Z'$ weakly northwest of $(a,b+1)$.  Hence, $\det(M')$ is a Fulton generator of $I_{\widetilde{A}}$.  But $z_{0,b+1}\mu = \init \det(M')$, and so $\mu \in \init I_{\widetilde{A}}:z_{0,j+1}\cdots z_{0,n+1}$, as desired.

    It remains to show \[
    (\init I_{\widetilde{A}}:z_{0,j+1}\cdots z_{0,n+1})+(z_{0,j+1}, \ldots, z_{0,n+1}) \subseteq \init I_A+(z_{0,1}, \ldots,z_{0,j-1}, z_{0,j+1}, \ldots,  z_{0,n+1}),
    \] for which it suffices to show $\init I_{\widetilde{A}}:z_{0,j+1}\cdots z_{0,n+1} \subseteq \init I_A+(z_{0,1}, \ldots,z_{0,j-1}, z_{0,j+1}, \ldots,  z_{0,n+1})$.  Fix a monomial $\nu \in \init I_{\widetilde{A}}:z_{0,j+1}\cdots z_{0,n+1}$.  First suppose $\nu \in \init I_{\widetilde{A}}$.  We may assume that $\nu$ is the leading term of some Fulton generator $g$ of $I_{\widetilde{A}}$.  Then there exist $(c,d) \in \ess(\widetilde{A})$ and $N'$ an $(\rk_{\widetilde{A}}(c,d)+1) \times (\rk_{\widetilde{A}}(c,d)+1)$ submatrix of $Z'$ weakly northwest of $(c,d)$ so that $g = \det(N')$.

    If $d<j$, then, either $\nu \in (z_{0,1}, \ldots, z_{0,j-1})$, which occurs if $N'$ involves row $0$, or $g$ is a Fulton generator of $I_A$ belonging to essential cell $(c,d)$, which occurs if $N'$ does not involve row $0$.  If $d > j$, then $\nu$ is equal to the product of the northeast entry of $N'$ with the lead term of a Fulton generators of $A$ belonging to essential cell $(c,d-1)$ of $A$.  Note that $d \neq j$ because $\widetilde{A}$ has no essential cells in column $j$.
    
    Finally, suppose $\nu \in (\init I_{\widetilde{A}}:z_{0,j+1}\cdots z_{0,n+1})\setminus \init I_{\widetilde{A}}$.  Because the variables $z_{0,k}$, $k \in [j+1,n+1]$, all belong to the same row, each minimal generator of $\init I_{\widetilde{A}}$ is divisible by at most one such $z_{0,k}$.  We may assume there exists $k\in [j+1,n+1]$ so that $z_{0,k}\nu$ is the product of antidiagonal entries of a submatrix of $Z'$ pertaining to an essential cell $(c,d)$ of $\widetilde{A}$ with $d>j$.  Then $\nu$ is the product of antidiagonal entries of a submatrix of $Z$ pertaining to essential cell $(c,d-1)$, which is to say that $\nu \in \init I_A$, completing the proof.
\end{proof}

We have now seen that $A$ is equidimensional if and only if $1 \oplus A$ is equidimensional and that $1 \oplus A$ Cohen--Macaulay implies $A$ Cohen--Macaulay.  A Macaulay2 computation confirms that, over the field of rational numbers, if $A \in \asm(n)$ with $n \leq 6$ and $A$ is Cohen--Macaulay, then  that $1 \oplus A$ is Cohen--Macaulay.  Based on this evidence, we make the following conjecture:

\begin{conjecture}\label{conj:CM}
    The ASM $A$ is Cohen--Macaulay if and only if $1 \oplus A$ is Cohen--Macaulay.
\end{conjecture}

The truth of Conjecture \ref{conj:CM} would have implications to any ASM obtained as a diagonal block sum of others.  Our next goal is to state and prove that implication, for which we require a couple of lemmas.

The following is a routine exercise, which can be found, for example, in \cite[Chapter 3]{Vil15}.

\begin{lemma}\label{lem:CMness-variables}
Let $S_1 = \mathbb{C}[x_1, \ldots, x_n]$, $S_2 = \mathbb{C}[y_1, \ldots, y_m]$, and $R = \mathbb{C}[x_1, \ldots, x_n,y_1,\ldots,y_m]$.  Suppose that $I_1$ is a proper homogeneous ideal of $S_1$ and that $I_2$ is a proper homogeneous ideal of $S_2$. Then $R/(I_1R+I_2R)$ is Cohen--Macaulay if and only if $S_1/I_1$ and $S_2/I_2$ are both Cohen--Macaulay.  The associated primes of $R/(I_1R+I_2R)$ are exactly those ideals of the form $P_1R+P_2R$ where $P_i$ is an associated prime of $I_i$.  In particular, $R/(I_1R+I_2R)$ is equidimensional if and only if $S_1/I_1$ is and $S_2/I_2$ are both equidimensional.
\end{lemma}

Recall that the \newword{support} of a monomial ideal $J$, denoted $\supp(J)$, is the set of variables dividing some minimal monomial generator of $J$.  If $J = \init I_w$ for some $w \in S_n$, the $\supp(J)$ is sometimes called the \emph{core} of $w$.

The following lemma is known to experts. We record it below for completeness.

\begin{lemma}\label{lem:antidiagonal-initial-ideal-above-main-antidiagonal}
If $A \in \ASM(n)$, then $\supp(\init I_A) \subseteq \{z_{i,j} \mid i+j \leq n\}$.  
\end{lemma}
\begin{proof}
By \cite[Proposition 3.11]{Wei17}, there exist $u_1, \ldots, u_k \in S_n$ such that $I_A = I_{u_1}+\cdots+I_{u_k}$.  By Proposition \ref{prop:initialIdealSplit}, $\init I_A = \init I_{u_1}+\cdots+\init I_{u_k}$.  Hence, it suffices to show that $\supp(\init I_w) \subseteq \{z_{i,j} \mid i+j \leq n\}$ for arbitrary $w \in S_n$.  This follows from the pipe dream description of $\init I_w$ in \cite{KM05}. 
\end{proof}

    \begin{example}\label{ex:core}
Consider $w = 31542$, whose Rothe diagram is below. The cells $(i,j)$ with $z_{i,j} \in \supp(\init I_w)$ are those with an orange line running across their antidiagonal. \[
\begin{tikzpicture}[x=1.5em,y=1.5em]
\draw[step=1,gray, thin] (0,0) grid (5,5);
\draw[color=black, thick](0,0)rectangle(5,5);
\filldraw [black](0.5,3.5)circle(.1);
\filldraw [black](1.5,0.5)circle(.1);
\filldraw [black](2.5,4.5)circle(.1);
\filldraw [black](3.5,1.5)circle(.1);
\filldraw [black](4.5,2.5)circle(.1);

\draw[thick, color=blue] (.5,0)--(.5,3.5)--(5,3.5);
\draw[thick, color=blue] (1.5,0)--(1.5,0.5)--(5,0.5);
\draw[thick, color=blue] (2.5,0)--(2.5,4.5)--(5,4.5);
\draw[thick, color=blue] (3.5,0)--(3.5,1.5)--(5,1.5);
\draw[thick, color=blue] (4.5,0)--(4.5,2.5)--(5,2.5);

\draw[step=1, orange, thin](0,4)--(1,5);
\draw[step=1, orange, thin](1,4)--(2,5);
\draw[step=1, orange, thin](0,2)--(3,5);
\draw[step=1, orange, thin](0,1)--(4,5);
\end{tikzpicture} \qedhere
\]
\end{example}

\begin{thm}\label{prop:codimOfDirectSum}
    (a) Fix $A_1 \in \asm(m)$ and $A_2 \in \asm(n)$, and set $A = A_1 \oplus A_2$.  For $u \in S_m$ and $v \in S_n$, the assignment $(u,v) \mapsto u \oplus v$ gives a bijection between $\Perm(A_1) \times \Perm(A_2)$ and $\Perm(A)$.  In particular, $\codim(X_A) = \codim(X_{A_1})+\codim(X_{A_2})$, and $A$ is equidimensional if and only if $A_1$ and $A_2$ are both equidimensional.

   (b) If $A$ is Cohen--Macaulay, then $A_1$ and $A_2$ are both Cohen--Macaulay.  If Conjecture \ref{conj:CM} is true, then the converse also holds.
\end{thm}
\begin{proof}
Consider $R = \mathbb{C}[z_{i,j} \mid i, j \in [m+n]]$ as the ambient polynomial ring of $I_A$.  Consider the subrings $S_1 = \mathbb{C}[z_{i,j} \mid i+j \leq m]$ and $S_2 = \mathbb{C}[z_{i,j} \mid i+j > m]$ of $R$.  Then $R = S_1 \otimes_\mathbb{C} S_2$.

    For $k \geq 1$, let $\mathbb{I}_k$ denote the $k \times k$ identity matrix, and observe that \[
    I_A = I_{A_1 \oplus \mathbb{I}_n}+I_{\mathbb{I}_m \oplus A_2}.
    \]  By Proposition \ref{prop:initialIdealSplit}, \[
\init I_A = \init I_{A_1 \oplus \mathbb{I}_n}+ \init I_{\mathbb{I}_m \oplus A_2}.
    \] By Lemma \ref{lem:antidiagonal-initial-ideal-above-main-antidiagonal}, $\init I_{A_1+\mathbb{I}_n}$ is supported only on variables $z_{i,j}$ with $i+j\leq m$.  Note that $\init I_{\mathbb{I}_m \oplus A_2}$ is supported only on $z_{i,j}$ with $i+j>m+1$.  Specifically, $\init I_{A_1 \oplus \mathbb{I}_n}$ and $\init I_{\mathbb{I}_m \oplus A_2}$ are supported on disjoint sets of variables; the former on a subset of the variables of $S_1$ and the latter on a subset of the variables of $S_2$.  

     (a) By Lemma \ref{lem:CMness-variables}, the associated primes $P$ of $\init I_A$ are exactly those ideals of the form $Q_1+Q_2$ where $Q_1$ is an associated prime of $\init I_{A_1 \oplus \mathbb{I}_n}$ and $Q_2$ is an associated prime of $\init I_{\mathbb{I}_m \oplus A_2}$.  
    
    Combining \Cref{prop:initialIdealSplit} and \Cref{prop:components-give-reduced-words}, the associated primes of $\init I_{A_1 \oplus \mathbb{I}_m}$ are exactly those ideals of the form $\init I_{u'}$ for some $u' \in \Perm(A_1 \oplus \mathbb{I}_n)$, and the associated primes of $\init I_{\mathbb{I}_m \oplus A_2}$ are exactly those of the form $\init I_{v'}$ for some $v' \in \Perm(\mathbb{I}_m \oplus A_2)$.  Because $I_{A_1 \oplus \mathbb{I}_n} = I_{A_1}R$, those permutations $u'$ are exactly those of the form $u \oplus \mathbb{I}_m$ for some $u \in \Perm(A_1)$.  By Proposition \ref{prop:permBij}, those permutations $v'$ are exactly those of the form $\mathbb{I}_m \oplus v$ for some $v \in \Perm(A_2)$. Then $(u,v) \in \Perm(A_1) \times \Perm(A_2)$ if and only if $I_{u \oplus v} = I_{u \oplus \mathbb{I}_n}+I_{\mathbb{I}_m \oplus v}$ is an associated prime of $I_A$, which is equivalent to $u \oplus v \in \Perm(A)$. This completes the proof of the bijection between $\Perm(A_1) \times \Perm(A_2)$ and $\Perm(A)$. 

    Because, for any $u \in S_m$ and $v \in S_n$, $\ell(u \oplus v) = \ell(u)+\ell(v)$, the codimension and equidimensionality claims now follow from Proposition \ref{prop:perm-set-decomposition}.
    
    (b) We turn to the Cohen--Macaulayness claim.  By Theorem \ref{thm: initIdealCM}, $I_A$, $I_{A_1 \oplus \mathbb{I}_n}$, and $I_{\mathbb{I}_m\oplus A_2}$ define Cohen--Macaulay quotient rings if and only if $\init I_A$,  $\init I_{A_1 \oplus \mathbb{I}_n}$, and $\init I_{\mathbb{I}_m\oplus A_2}$ do, respectively.  By Proposition \ref{lem:CMness-variables}, $\init I_{A_1 \oplus \mathbb{I}_n}+\init I_{\mathbb{I}_m\oplus A_2}$ defines a Cohen--Macaulay quotient if and only if both $\init I_{A_1 \oplus \mathbb{I}_n}$ and $\init I_{\mathbb{I}_m\oplus A_2}$ do.  Because $I_{A_1 \oplus \mathbb{I}_n} = I_{A_1}R$, $I_{A_1 \oplus \mathbb{I}_n}$ and $I_{A_1}$ define Cohen--Macaulay quotients or not alike.  By \Cref{topPermRow}, if $I_{\mathbb{I}_m\oplus A_2}$ defines a Cohen--Macaulay quotient ring, then $I_{A_2}$ does.  If Conjecture \ref{conj:CM} is true, the converse is true as well.
\end{proof}

In Proposition \ref{prop:codimOfDirectSum}, one may alternatively establish the codimension and equidimensionality claims by appealing to Lemma \ref{lem:CMness-variables} together with the fact that codimension does not change under Gr\"obner degeneration (see, e.g., \cite[Chapter 15]{Eis95}).

\begin{prop}
    If $A$ admits a block-matrix decomposition of the form
\[A = \left(\begin{array}{@{}c|c@{}}
0 & A_1\\\hline
A_2 & 0\\
  \end{array}\right),
\]
where both $A_1 \in \asm(m)$ and $A_2 \in \asm(n)$ are ASMs, then $A$ is unmixed (respectively, Cohen--Macaulay) if and only if $A_1$ and $A_2$ are both unmixed (respectively, Cohen--Macaulay).
\end{prop} 
\begin{proof}
    Let $J = (z_{i,j} \mid i \in [m], j \in [n])$.  For a suitable choice of indexing of the variables, $I_A = J+I_{A_1}+I_{A_2}$, where $J$, $I_{A_1}$, and $I_{A_2}$ involve pairwise disjoint sets of variables.  Hence the result follows from \Cref{lem:CMness-variables}.
\end{proof}

In the next section, we will see through \Cref{containment-restrictions} that the special case of a block sum decomposition is not quite so specialized as it may seem in the sense that ASMs do not easily contain one another.

\section{Pattern avoidance for ASMs}

In this section, we will consider a natural extension of the notion of pattern avoidance in $S_n$ to $\asm(n)$. Permutation pattern avoidance, its own area of study within combinatorics, has been shown to govern desirable algebro-geometric properties within Schubert calculus (see, e.g., \cite{WY06, KMY09, Kle23} and quite importantly \cite{LS97} together with a survey of its consequences \cite{AB16}).  Properties that are governed by pattern avoidance occur asymptotically with probability $0$.  As we saw in \Cref{KM-vert-decomp-table}, the property of being Knutson--Miller vertex decomposable becomes increasingly rare.  So, too, it turns out does Cohen--Macaulayness, which we see now in \Cref{fig:CMData}.

\begin{figure}[h]
\begin{center}
\begin{tabular}{|c|c|c|c|} 
 \hline
 $n$ & \# CM ASMs & \# not CM ASMs & Percent CM\\ 
 \hline
 4 & 39 & 3 & 92.9\% \\
 \hline
 5 & 328 & 101 & 76.5\%\\
 \hline
 6 & 4028 & 3408 & 54.2\% \\
 \hline
  7 & 70,194 & 148,154 & 32.1\%\\
 \hline
\end{tabular}
\caption{Counts of Cohen--Macaulay and non-Cohen--Macaulay elements of $\asm(n)$.}
\label{fig:CMData}
\end{center}
\end{figure}

It is for this reason that we are motivated to consider pattern avoidance in $\asm(n)$.

\begin{definition}\label{defn:contains} Let $A$ and $A'$ be ASMs.  If $A'$ is a submatrix of $A$, then we will say that $A$ \newword{contains} $A'$.  Otherwise, we will say that $A$ \newword{avoids} $A'$.  
\end{definition}

These definitions coincide with the usual uses of ``contains" and ``avoids" in the sense of permutation pattern avoidance when $A$ and $A'$ are permutation matrices.

We will first use an example to show that, perhaps surprisingly, this extension of pattern avoidance to $\asm(n)$ given above does \emph{not} govern Cohen--Macaulayness of ASM varieties.  We will then prove that, nevertheless, there is still some behavior that is correctly understood via pattern avoidance.  Specifically, we will describe configurations whose containment is an obstruction to unmixedness.  Our goal in presenting this information is to provide motivation and context for future work to consider other extensions of pattern avoidance from $S_n$ to $\asm(n)$ that could do more work to capture the phenomena documented here, perhaps those beginning from a more geometric perspective as in \cite{BB03} or \cite{BS98}.

There are other notions of pattern avoidance for ASMs that have emerged in the literature on enumerative combinatorics (e.g., \cite{JL07, ACG11, BSS25}).  To the best of the authors' knowledge, these have not been shown to have algebro-geometric interpretations.  

We first consider the relationship between Cohen--Macaulayness and ASM pattern containment.  We have previously seen, in \Cref{ex:permBij} and \Cref{ex:widetilde{A}}, that a non-Cohen--Macaulay ASM can contain a Cohen--Macaulay ASM even when the larger ASM is obtained via the construction $\widetilde{A}(i,j)$.  What is more surprising is that, as \Cref{ex:CM-contains-non} shows, a Cohen--Macaulay ASM may contain a non-Cohen--Macaulay ASM.

\begin{example}\label{ex:CM-contains-non}
    Let $A = \begin{pmatrix} 
    0 & 0 & {\color{blue}0} & 1 & 0 & 0\\
    0 & 0 & {\color{blue}0} & 0 & 1 & 0\\
    0 & 1 & {\color{blue}0} & 0 & -1 & 1\\
    {\color{blue}0} & {\color{blue}0} & {\color{blue}1} & {\color{blue}0} & {\color{blue}0} & {\color{blue}0}\\
    1 & 0 & {\color{blue}0} & -1 & 1 & 0\\
    0 & 0 & {\color{blue}0} & 1 & 0 & 0
   \end{pmatrix} $ and $B = \begin{pmatrix} 
    0 & 0  & 1 & 0 & 0\\
    0 & 0  & 0 & 1 & 0\\
    0 & 1  & 0 & -1 & 1\\
    1 & 0  & -1 & 1 & 0\\
    0 & 0  & 1 & 0 & 0
   \end{pmatrix} $.  Then $A$ is Cohen--Macaulay while $B$ is not even unmixed, even though $B$ is obtained from $A$ by deletion of row $4$ and column $3$.  That is, $A = \widetilde{B}(4,3)$.  Indeed, $\Perm(B) = \{45213, 34512, 35241\}$, whose first and third elements have length $7$ while the second has length $6$. In contrast to the situation of Proposition \ref{topPermRow}, some elements of $\Perm(A) = \{562314,462513,456213,462351\}$ fail to have a $1$ in row $4$ and column $3$, for example $462513$.  Moreover, $|\Perm(A)| = 4 \neq 3 = |\Perm(3)|$.
\end{example}

We provide one more piece of evidence of a shortcoming of this notion of pattern avoidance.  We will show that the conditions implied by the containment of one ASM in another are rather restrictive.  From that standpoint, one may worry that ASMs do not contain each other frequently enough to do the heavy lifting of recording obstructions to important algebro-geometric properties that become increasingly rare.

\begin{prop}
\label{containment-restrictions}
Suppose $A \in \ASM(n)$ has a submatrix $A' \in \ASM(n-k)$.  Let $W = \{r_1,\dots, r_k\}$ be the ordered set of indices of rows of $A$ that do not intersect $A'$ and $C = \{c_1, \dots, c_k\}$ be the ordered set of indices of columns of $A$ that do not intersect $A'$. Then the following hold:

\begin{enumerate}
    \item For all $i \in [k]$, $A_{r_i, c}=0$ if $c < c_1$ or $c>c_k$ and $A_{r,c_i}=0$ if $r < r_1$ or $r > r_k$.
    \item $\sum_{r \in W, c \in C} A_{r,c} = k$. In particular, there are at least $k$ pairs $(r_i,c_j)$ such that $A_{r_i,c_j} = 1$.
\end{enumerate}
\end{prop}
\begin{proof}
Throughout this proof, for a visualization of the various regions discussed, see \Cref{fig-for-containment-restrictions}.

(1) Consider first the case $c < c_1$. For the sake of contradiction, suppose there exists some $i \in [k]$ and $c<c_1$ with $A_{r_i,c} \neq 0$.  Assume that $c$ has been chosen minimally. Because the first nonzero entry in each row of an ASM must be $1$, $A_{r_i,c} = 1$. The sum of the entries in column $c$ of $A$ is $1$, and also the sum of the entries in column $c$ of $A'$ is $1$ because both $A$ and $A'$ are ASMs. Thus, $\sum_{j \in [k]} A_{r_j, c} = 0$, and so there must be some $A_{r_j,c} = -1$.  But, by minimality of $c$, $A_{r_j,c}$ must be the first nonzero entry in row $r_j$ and therefore cannot be $-1$.

The arguments for the other cases of (1) are symmetric, and so we omit them.

(2) Because $A'$ is an ASM, $\sum_{j \in [k]} A_{r_j,c} = 0$ for each $c \notin C$ and $\sum_{i \in [k]} A_{r,c_i} = 0$ for each $r \notin W$. Note also that, because $A \in \asm(n)$ and $A' \in \asm(n-k)$, we have, respectively, \[
n = \sum_{r \in [n], c \in [n] } A_{r,c} \mbox{ and } n-k = \sum_{r \notin R, c \notin C } A_{r,c}.
\]  Hence,

\begin{align*}
n = \sum_{{r \in [n], c \in [n] }} A_{r,c} &= \sum_{r \notin W, c \notin C } A_{r,c} + \sum_{r \in W, c \notin C } A_{r,c} + \sum_{ r \notin W, c \in C } A_{r,c} + \sum_{r \in W, c \in C } A_{r,c}\\
 &= (n-k)+0+0+\sum_{r \in W, c \in C } A_{r,c},
\end{align*} from which the result follows.
\end{proof}

\begin{figure}[h]
\begin{center}
    \begin{tikzpicture}
        \fill[pink!70] (0,0) rectangle (4,4);
        
        \fill[yellow!70] (0,1.5) rectangle (4,1.75);
        \fill[yellow!70] (0,1) rectangle (4,1.25);
          \fill[yellow!70] (0,2.75) rectangle (4,3);
        \node at (1.15,4.25){$c_1$};
        \node at (2.15,4.25){$c_2$};
        \node at (3.35,4.25){$c_3$};

        \fill[yellow!70] (3.25,0) rectangle (3.5,4);
        \fill[yellow!70] (2,0) rectangle (2.25,4);
        \fill[yellow!70] (1,0) rectangle (1.25,4); 
         \node at (4.25, 1.15){$r_3$};
        \node at (4.25, 1.65){$r_2$};
        \node at (4.25, 2.85){$r_1$};
        
        \draw [thick] (0,0) rectangle (4,4);
        \draw [thick] (1,3) rectangle (3.5,1);

        \node at (1.15,2.3){\textcolor{green}{$\ast$}};
        \node at (2.15,2.3){\textcolor{green}{$\ast$}};
        \node at (3.35,2.3){\textcolor{green}{$\ast$}};
        \node at (1.15,1.15){$\star$};
        \node at (1.15,1.65){$\star$};
        \node at (1.15,2.85){$\star$};
        \node at (2.15,1.15){$\star$};
        \node at (2.15,1.65){$\star$};
        \node at (2.15,2.85){$\star$};
        \node at (3.35,1.15){$\star$};
        \node at (3.35,1.65){$\star$};
        \node at (3.35,2.85){$\star$};
        \node at (-0.5,2) {\Large $A= $ \hspace{3cm}};
        
    \end{tikzpicture}
\end{center}
\caption{A visualization of the various regions discussed in \Cref{containment-restrictions} with $k=3$. The matrix $A'$ is the submatrix of $A$ consisting of its entries in the pink regions.  The rows and columns that are removed from $A$ to obtain $A'$ are shaded in yellow.  The entries in the yellow strips outside of the inner box are all $0$.  The sum, for example, of the three entries marked \textcolor{green}{$\ast$} is $0 = \sum_{j \in [3]} A_{r,c_j}$ for some $r \notin W$, i.e., for some non-yellow row $r$. The sum of the nine entries marked with a $\star$ is $\sum_{r \in W, c \in C} A_{r,c} = k=3$.}
\label{fig-for-containment-restrictions}
\end{figure}
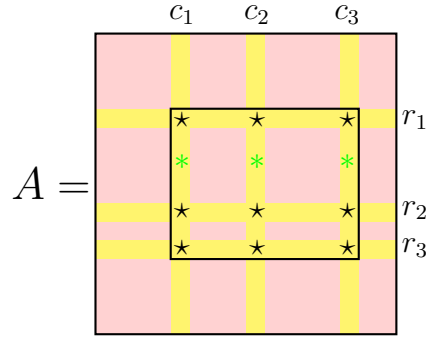

\begin{cor}
    If $A' \in \asm(n-k)$ embeds in $A \in \asm(n)$ as a block in the northwest corner, i.e. $A = \left(\begin{array}{@{}c|c@{}}
    A' & * \\\hline
    *  & * 
  \end{array}\right)$, then $A = A' \oplus B$ for some $B \in \asm(k)$. 
\end{cor}
\begin{proof}
    We use \Cref{containment-restrictions}, where $W = C = [n-k+1,n]$, to see that $A = A' \oplus B$ for some $k \times k$ matrix $B$ with entries in $\{0, 1, -1\}$.  From the direct sum decomposition, we see that all nonzero entries of $A$ in each row (resp., column) $i \in [n-k+1,n]$ occur in columns (resp., rows) $[n-k+1,n]$.  Hence, the nonzero entries of each row (resp., column) of $B$ alternate in sign and sum to $1$.  Therefore, $B \in \asm(k)$.
\end{proof}

Although the notion of pattern avoidance given in \Cref{defn:contains} has significant limitations, it should not be entirely ignored.  As an example of its capacity to encode some valuable information, we will show in \Cref{prop:not-equidimensional} that containing an ASM belonging to a particular family is adequate to prevent unmixedness.  We present this as evidence that pattern avoidance should not be entirely abandoned as a means to understand the algebra and geoemtry of ASM varieties.  We begin with a lemma.

\begin{lemma}
\label{lem:findminprimes}
Let $I$ be a squarefree monomial ideal, and let $S$ be a subset of the support of $I$. Let $P = (S)$. Then $P$ is a minimal prime of $I$ if and only if both of the following two conditions hold:

\begin{enumerate}
    \item [(i)] For each monomial $\mu \in I$, at least one element of the support of $\mu$ is an element of $S$.
    
    \item [(ii)] For every $x \in S$, there exists a monomial $\nu \in I$ so that $x$ is in the support of $\nu$ and $x$ is the only element of $S$ in the support of $\nu$.
\end{enumerate}
\end{lemma}

\begin{proof}
Clearly $(i)$ holds if and only if $I \subseteq P$.  It is also clear that $P$ is prime.  It remains to show $(ii)$ holds if and only if that $I \subseteq P' \subseteq P$ for a prime ideal $P'$ implies $P = P'$.

Because the minimal primes of monomial ideals are monomial ideals, it suffices to show that $(ii)$ holds if and only if that $I \subseteq (S') \subseteq P$ for a subset $S'$ of $S$ implies $P = (S')$.  But $(ii)$ is false if and only if there is a proper subset $S'$ of $S$ so that condition $(i)$ holds for $S'$, which is true if and only if $I \subseteq (S') \subset (S)$.  Specifically, if $(ii)$ is false and $x \in S$ is the violating element, set $S'  = S \setminus \{x\}$.  Conversely, if such an $S' \subset S$ exists, take $x$ to be any element of $S \setminus S'$.
\end{proof}

Using Lemma \ref{lem:findminprimes}, we can show that various families of ASMs are not equidimensional (and therefore not Cohen-Macaulay) by finding two minimal primes which we demonstrate to be of different heights. We give such a class of ASMs now.

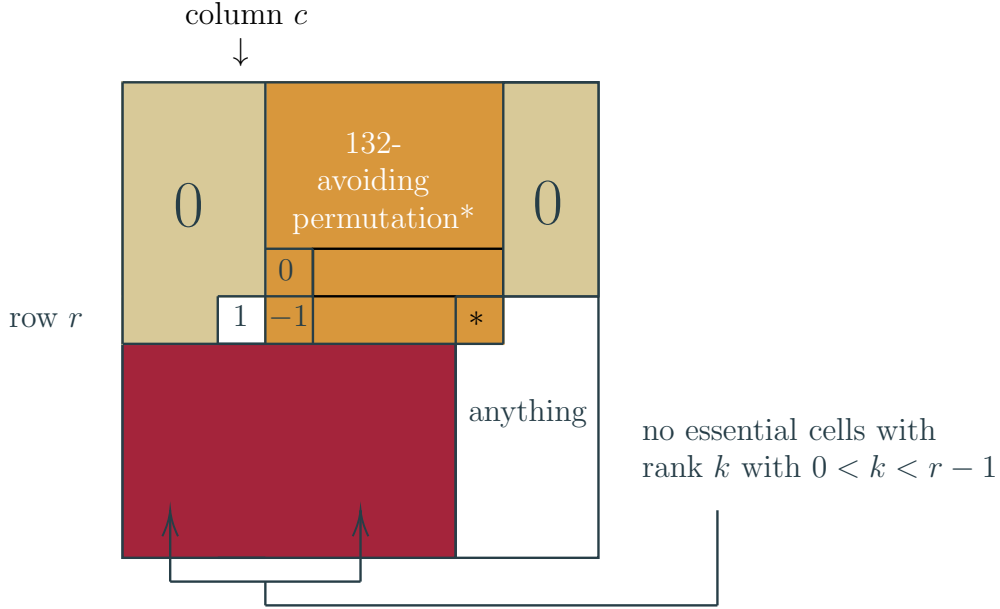
\begin{figure}[h]
    \centering
    \tikzset{every picture/.style={line width=0.9pt}} 

\begin{tikzpicture}[x=0.9pt,y=0.9pt,yscale=-1,xscale=1]

\draw  [color={rgb, 255:red, 39; green, 62; blue, 71 }  ,draw opacity=1 ] (160,60) -- (360,60) -- (360,260) -- (160,260) -- cycle ;
\draw    (220,60) -- (220,130) ;
\draw    (320,60) -- (320,150) ;
\draw    (300,170) -- (240,170) ;
\draw    (300,150) -- (300,170) ;
\draw    (300,150) -- (320,150) ;
\draw    (320,150) -- (360,150) ;
\draw  [color={rgb, 255:red, 39; green, 62; blue, 71 }  ,draw opacity=1 ][fill={rgb, 255:red, 164; green, 36; blue, 59 }  ,fill opacity=1 ] (160,170) -- (300,170) -- (300,260) -- (160,260) -- cycle ;
\draw  [color={rgb, 255:red, 216; green, 151; blue, 60 }  ,draw opacity=1 ][fill={rgb, 255:red, 216; green, 151; blue, 60 }  ,fill opacity=1 ] (220,60) -- (320,60) -- (320,170) -- (220,170) -- cycle ;
\draw  [draw opacity=0][fill={rgb, 255:red, 216; green, 151; blue, 60 }  ,fill opacity=1 ] (240,130) -- (320,130) -- (320,150) -- (240,150) -- cycle ;
\draw [color={rgb, 255:red, 39; green, 62; blue, 71 }  ,draw opacity=1 ]   (180,270) -- (180,242) ;
\draw [shift={(180,240)}, rotate = 90] [color={rgb, 255:red, 39; green, 62; blue, 71 }  ,draw opacity=1 ][line width=0.75]    (10.93,-3.29) .. controls (6.95,-1.4) and (3.31,-0.3) .. (0,0) .. controls (3.31,0.3) and (6.95,1.4) .. (10.93,3.29)   ;
\draw [color={rgb, 255:red, 39; green, 62; blue, 71 }  ,draw opacity=1 ]   (260,270) -- (260,242) ;
\draw [shift={(260,240)}, rotate = 90] [color={rgb, 255:red, 39; green, 62; blue, 71 }  ,draw opacity=1 ][line width=0.75]    (10.93,-3.29) .. controls (6.95,-1.4) and (3.31,-0.3) .. (0,0) .. controls (3.31,0.3) and (6.95,1.4) .. (10.93,3.29)   ;
\draw [color={rgb, 255:red, 39; green, 62; blue, 71 }  ,draw opacity=1 ]   (180,270) -- (260,270) ;
\draw [color={rgb, 255:red, 39; green, 62; blue, 71 }  ,draw opacity=1 ]   (220,270) -- (220,280) ;
\draw [color={rgb, 255:red, 39; green, 62; blue, 71 }  ,draw opacity=1 ]   (220,280) -- (410,280) ;
\draw [color={rgb, 255:red, 39; green, 62; blue, 71 }  ,draw opacity=1 ]   (410,240) -- (410,280) ;
\draw  [color={rgb, 255:red, 216; green, 201; blue, 152 }  ,draw opacity=1 ][fill={rgb, 255:red, 216; green, 201; blue, 152 }  ,fill opacity=1 ] (160,60) -- (220,60) -- (220,150) -- (160,150) -- cycle ;
\draw  [color={rgb, 255:red, 216; green, 201; blue, 152 }  ,draw opacity=1 ][fill={rgb, 255:red, 216; green, 201; blue, 152 }  ,fill opacity=1 ] (160,150) -- (200,150) -- (200,170) -- (160,170) -- cycle ;
\draw  [color={rgb, 255:red, 39; green, 62; blue, 71 }  ,draw opacity=1 ][fill={rgb, 255:red, 216; green, 201; blue, 152 }  ,fill opacity=1 ] (320,60) -- (360,60) -- (360,150) -- (320,150) -- cycle ;
\draw [color={rgb, 255:red, 39; green, 62; blue, 71 }  ,draw opacity=1 ]   (160,170) -- (300,170) ;
\draw [color={rgb, 255:red, 39; green, 62; blue, 71 }  ,draw opacity=1 ]   (200,150) -- (200,170) ;
\draw [color={rgb, 255:red, 39; green, 62; blue, 71 }  ,draw opacity=1 ]   (220,60) -- (220,170) ;
\draw [color={rgb, 255:red, 39; green, 62; blue, 71 }  ,draw opacity=1 ]   (300,170) -- (300,260) ;
\draw [color={rgb, 255:red, 39; green, 62; blue, 71 }  ,draw opacity=1 ]   (160,60) -- (360,60) ;
\draw [color={rgb, 255:red, 39; green, 62; blue, 71 }  ,draw opacity=1 ]   (200,150) -- (220,150) ;
\draw [color={rgb, 255:red, 39; green, 62; blue, 71 }  ,draw opacity=1 ]   (220,130) -- (240,130) ;
\draw [color={rgb, 255:red, 39; green, 62; blue, 71 }  ,draw opacity=1 ]   (240,130) -- (240,170) ;
\draw [color={rgb, 255:red, 39; green, 62; blue, 71 }  ,draw opacity=1 ]   (320,170) -- (320,150) ;
\draw [color={rgb, 255:red, 39; green, 62; blue, 71 }  ,draw opacity=1 ]   (220,150) -- (240,150) ;
\draw [color={rgb, 255:red, 39; green, 62; blue, 71 }  ,draw opacity=1 ]   (160,60) -- (160,170) ;
\draw [color={rgb, 255:red, 39; green, 62; blue, 71 }  ,draw opacity=1 ]   (200,260) -- (220,260) ;
\draw [color={rgb, 255:red, 39; green, 62; blue, 71 }  ,draw opacity=1 ]   (300,170) -- (320,170) ;
\draw [color={rgb, 255:red, 39; green, 62; blue, 71 }  ,draw opacity=1 ]   (300,150) -- (320,150) ;
\draw [color={rgb, 255:red, 39; green, 62; blue, 71 }  ,draw opacity=1 ]   (300,170) -- (300,150) ;

\draw (180,101.4) node [anchor=north west][inner sep=0.75pt]  [font=\huge]  {$\textcolor[rgb]{0.15,0.24,0.28}{0}$};
\draw (331,100.4) node [anchor=north west][inner sep=0.75pt]  [font=\huge]  {$\textcolor[rgb]{0.15,0.24,0.28}{0}$};
\draw (224,133.4) node [anchor=north west][inner sep=0.75pt]    {$\textcolor[rgb]{0.15,0.24,0.28}{0}$};
\draw (230,80) node [anchor=north west][inner sep=0.75pt]  [color={rgb, 255:red, 255; green, 255; blue, 255 }  ,opacity=1 ] [align=left] {\begin{minipage}[lt]{61.42pt}\setlength\topsep{0pt}
\begin{center}
\textcolor[rgb]{1,1,1}{$\displaystyle 132$}\textcolor[rgb]{1,1,1}{-avoiding}\\\textcolor[rgb]{1,1,1}{permutation*}
\end{center}

\end{minipage}};
\draw (377,200) node [anchor=north west][inner sep=0.75pt]   [align=left] {\textcolor[rgb]{0.15,0.24,0.28}{no essential cells with}\\\textcolor[rgb]{0.15,0.24,0.28}{rank $k$ with }\textcolor[rgb]{0.15,0.24,0.28}{$\displaystyle 0< k< r-1$}};
\draw (205,152.4) node [anchor=north west][inner sep=0.75pt]    {$\textcolor[rgb]{0.15,0.24,0.28}{1}$};
\draw (220,152.4) node [anchor=north west][inner sep=0.75pt]    {$\textcolor[rgb]{0.15,0.24,0.28}{-1}$};
\draw (111,156) node [anchor=north west][inner sep=0.75pt]   [align=left] {\textcolor[rgb]{0.15,0.24,0.28}{row }\textcolor[rgb]{0.15,0.24,0.28}{$\displaystyle r$}};
\draw (304,192) node [anchor=north west][inner sep=0.75pt]   [align=left] {\textcolor[rgb]{0.15,0.24,0.28}{anything}};
\draw (304,155) node [anchor=north west][inner sep=0.75pt]    {$\ast $};
\draw (205,40) node [anchor=north west][inner sep=0.75pt]{$\downarrow$};
\draw (185,25) node [anchor=north west][inner sep=0.75pt]{column $c$};

\end{tikzpicture}

    \caption{(*) Pick a $132$-avoiding permutation whose bottom left boxes is $0$. Then replace the bottom left entry with $-1$ and the bottom right with any entry allowable by the definition of ASM.}
    \label{fig:not_unmixed_family}
\end{figure}

In the proof of the following proposition, we will use the total order $<$ on $[n] \times [n]$ given by $(i,j) < (i',j')$ if $i < i'$ or if $i = i'$ and $j < j'$.  That is, \[
(1,1) < (1,2) < \cdots < (1,n) < (2,1) < \cdots < (n,n-1) < (n,n).
\]

We will use $<$ as the total order for an inductive argument.  Note that $<$ is not related to a total order on the monomials determining the antidiagonal initial ideal $\init I_A$ of the ASM $A$.

\begin{prop}\label{prop:not-equidimensional}
Suppose that $A = (a_{ij}) \in \ASM(n)$ satisfies the following properties:

\begin{enumerate}
    \item There exists $(r,c) \in [n-1] \times [n-2]$ so that the submatrix $B$ of $A$ consisting of rows $\{r-1,r\}$ and columns $\{c,c+1\}$ has the form $B = \ytableausetup{centertableaux}
\begin{ytableau}
 0 & 0 \\
1 & -1
\end{ytableau}$.
    
    \item  If $i \leq r$ and $j \leq c$ and $(i,j) \neq (r,c)$, then $a_{ij} = 0$.
    
    \item If $(i,j) \in \ess(A)$ and $(i,j) \neq (r,c+1)$, then $\rk_A(i,j) = 0$ or $\rk_A(i,j) \geq r-1$.
    
    \item $A$ has no essential cell in column $c$.
    
Then $A$ is not equidimensional.
\end{enumerate}
\label{prop:badblock}
\end{prop}

\begin{proof}
Recall that $A$ is equidimensional if and only if all minimal primes of $\init I_A$ are of the same height.  We will construct two different minimal primes of $\init I_A$ and show that they have different heights.

For $(i,j) \in [n] \times [n]$, let let \[
Z_{i,j} = \{z_{i',j'} \mid i' \leq i, j' \leq j\} \setminus \{z_{i',j'} \mid (i',j')\in \Dom(A)\}.
\] Let $F_{i,j}$ be the set of Fulton generators of $I_A$ satisfying $\supp(\init F_{i,j}) \subseteq Z_{i,j}$.  Set $I_{i,j} = (\init f \mid f \in F_{i,j})$. 
Then $\init I_A = I_{n,n}+(z_{a,b} \mid (a,b) \in \Dom(A))$, and no term of any generator of $I_{n,n}$ is divisible by any $z_{a,b}$ with $(a,b) \in \Dom(A)$. Hence, every minimal $P$ prime of $\init I_A$ will be the sum of a minimal prime $Q$ of $I_{n,n}$ with $(z_{a,b} \mid (a,b) \in \Dom(A))$, and $\codim(P) = \codim(Q)+|\Dom(A)|$. Thus, it suffices to show that $I_{n,n}$ has two or more minimal primes of different heights.

We will show that, for all $(r,n) \leq (i,j) \leq (n,n)$, $I_{i,j}$ has two or more minimal primes of different heights. We will proceed by induction. We begin by finding two different minimal primes of $I_{r,n}$ which have different heights. 

Before giving the formal argument, we recommend a visualization: Draw the grid of points representing the variables $z_{i,j}$ in a generic $n \times n$ matrix. For each minimal generator $\mu$ of $I_{r,n}$, we draw a wire connecting the set of points corresponding to variables in the support of $\mu$. Figure \ref{fig:smallbadblock} shows such a drawing for a possible $I_{r,n}$. 

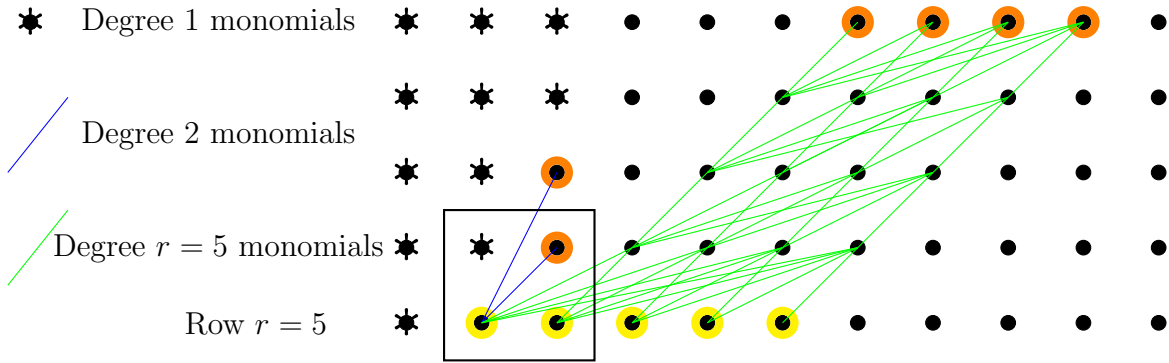
\begin{figure}[htp]
\begin{center}
\begin{tikzpicture}
\foreach \x in {0, 1, 2, ..., 10} {
    \foreach \y in {0, 1, 2, 3, 4} {
        \node[draw=black, fill=black, circle, scale=0.5] at (\x, \y) {};
    }
} 
\foreach \x in {1,2,3,4,5} {
\draw [yellow, thick](\x,0)circle(.2);
\draw [yellow, thick](\x,0)circle(.18);
\draw [yellow, thick](\x,0)circle(.16);
\draw [yellow, thick](\x,0)circle(.14);
\draw [yellow, thick](\x,0)circle(.12);
}

\foreach \y in {1,2} {
\draw [orange, thick](2,\y)circle(.2);
\draw [orange, thick](2,\y)circle(.18);
\draw [orange, thick](2,\y)circle(.16);
\draw [orange, thick](2,\y)circle(.14);
\draw [orange, thick](2,\y)circle(.12);
}

\foreach \x in {6,7,8,9} {
\draw [orange, thick](\x,4)circle(.2);
\draw [orange, thick](\x,4)circle(.18);
\draw [orange, thick](\x,4)circle(.16);
\draw [orange, thick](\x,4)circle(.14);
\draw [orange, thick](\x,4)circle(.12);
}

\foreach \y in {0,1,2,3,4} {
\node[scale = 2] at (0,\y){$\ast$};
}

\foreach \y in {1,2,3,4} {
\node[scale = 2] at (1,\y){$\ast$};
}

\foreach \y in {3,4} {
\node[scale = 2] at (2,\y){$\ast$};
}

\draw[color=blue] (1,0)--(2,1);
\draw[color=blue] (1,0)--(2,2);
\draw[color=green] (1,0)--(3,1);
\draw[color=green] (1,0)--(4,1);
\draw[color=green] (1,0)--(5,1);
\draw[color=green] (1,0)--(6,1);
\draw[color=green] (2,0)--(3,1);
\draw[color=green] (2,0)--(4,1);
\draw[color=green] (2,0)--(5,1);
\draw[color=green] (2,0)--(6,1);
\draw[color=green] (3,0)--(4,1);
\draw[color=green] (3,0)--(5,1);
\draw[color=green] (3,0)--(6,1);
\draw[color=green] (4,0)--(5,1);
\draw[color=green] (4,0)--(6,1);
\draw[color=green] (5,0)--(6,1);
\draw[color=green] (3,1)--(4,2);
\draw[color=green] (3,1)--(5,2);
\draw[color=green] (3,1)--(6,2);
\draw[color=green] (3,1)--(7,2);
\draw[color=green] (4,1)--(5,2);
\draw[color=green] (4,1)--(6,2);
\draw[color=green] (4,1)--(7,2);
\draw[color=green] (5,1)--(6,2);
\draw[color=green] (5,1)--(7,2);
\draw[color=green] (6,1)--(7,2);
\draw[color=green] (4,2)--(5,3);
\draw[color=green] (4,2)--(6,3);
\draw[color=green] (4,2)--(7,3);
\draw[color=green] (4,2)--(8,3);
\draw[color=green] (5,2)--(6,3);
\draw[color=green] (5,2)--(7,3);
\draw[color=green] (5,2)--(7,3);
\draw[color=green] (6,2)--(7,3);
\draw[color=green] (6,2)--(8,3);
\draw[color=green] (7,2)--(8,3);
\draw[color=green] (5,3)--(6,4);
\draw[color=green] (5,3)--(7,4);
\draw[color=green] (5,3)--(8,4);
\draw[color=green] (5,3)--(9,4);
\draw[color=green] (6,3)--(7,4);
\draw[color=green] (6,3)--(8,4);
\draw[color=green] (6,3)--(9,4);
\draw[color=green] (7,3)--(8,4);
\draw[color=green] (7,3)--(9,4);
\draw[color=green] (8,3)--(9,4);

\node at (-2,0) {Row $r=5$};
 \node[draw=black, fill=black, circle, scale=0.5] at (-5, 4){};
 \node[scale = 2] at (-5,4){$\ast$};
\node at (-2.5,4) {Degree $1$ monomials};
\node at (-2.5,2.5) {Degree $2$ monomials};
\draw[color=blue] (-5.3,2)--(-4.5,3);
\node at (-2.5,1) {Degree $r=5$ monomials};
\draw[color=green] (-5.3,0.5)--(-4.5,1.5);

\draw [thick] (0.5,-0.5) rectangle (2.5,1.5);
\end{tikzpicture}
\end{center}
    \caption{The $r \times n = 5 \times 11$ grid corresponding to the variables $z_{i,j}$ for $(i,j) \in [r] \times [n]$. Elements of $\Dom(A)$ are denoted by stars. The elements of $O_{r,n}$ are circled in orange, and those of $Y_{r,n}$ are circled in yellow.  Minimal generators of $I_{r,n}$ are indicated with blue and green wires connecting elements of $Y_{r,n}$ to elements of $O_{r,n}$.  In this example, $c=2$.  The submatrix $B$ is boxed.}
    \label{fig:smallbadblock}
\end{figure}

Given such a drawing for $I_{r,n}$, we draw an orange circle around the northeast vertex of each minimal monomial generator of $I_{r,n}$. Call this set of vertices $O_{r,n}$. Now we draw a yellow circle around the southwest vertex of each minimal monomial generator of $I_{r,n}$, and call this set of vertices $Y_{r,n}$. 

More formally, if $\mu = \init f$ for some $f \in F_{r,n}$, then, with respect to the  order $<$ on $[n] \times [n]$, let 
\[
O(\mu) = \min\{(a,b) \mid z_{a,b} \in \supp(\mu)\}.
\] and \[
O_{r,n} = \{z_{O(\mu)} \mid \mu = \init f, f \in F_{r,n}\}.
\]  Similarly, let \[
Y(\mu) = \max\{(a,b) \mid z_{a,b} \in \supp(\mu)\}
\] and \[
Y_{r,n} = \{z_{Y(\mu)} \mid \mu = \init f, f \in F_{r,n}\}.
\]

Because each minimal generator $\mu$ of $\init I_A$ of degree $k$ is formed by the product of variables along the antidiagonal of a generic matrix, the $k$ variables in the support of $\mu$ have distinct row indices and distinct column indices. The generators of $I_{r,n}$ are supported only on variables in the first $r$ rows of a generic matrix, and so every generator of degree $r$ in $I_{r,n}$ is supported on some variable in row $1$ and some variable in row $r$ (and on some variable in each row in between).

By property (3) of the hypotheses, each $f \in F_{r,n}$ is either of degree $r$ or is determined by the essential cell $(r,c+1)$. By property (2) of the hypotheses, any $f \in F_{r,n}$ determined by the essential cell $(r,c+1)$ must be the determinant of a $2 \times 2$ submatrix whose bottom row is $[z_{r,c}$ $z_{r,c+1}]$. Thus, all minimal generators of $I_{r,n}$ are divisible by some variable in row $r$.  Hence, we have $Y_{r,n} = \{z_{r,c}, \ldots, z_{r, c+i}\}$ for some $i \geq 0$. 

Similarly, $O_{r,n}$ will consist of the set of variables $\{z_{1,c+r},\ldots, z_{1,c+i+r-1}\}$ together with at least two more variables in column $c+1$, determined by the degree $2$ minimal generators of $I_{r,n}$. Hence, $|O_{r,n}| \geq |Y_{r,n}|+1$.

By construction, both $Y_{r,n}$ and $O_{r,n}$ satisfy the conditions of Lemma \ref{lem:findminprimes}, and so both $(Y_{r,n})$ and $(O_{r,n})$ are minimal primes of $I_{r,n}$. Now $\codim((O_{r,n})) = |O_{r,n}| >|Y_{r,n}| = \codim((Y_{r,n}))$, and so $I_{r,n}$ is not height unmixed.

For $(i,j)$ satisfying $(r,n)<(i,j) \leq (n,n)$, we will define sets $Y_{i,j}$ and $O_{i,j}$ inductively. We will then argue, for all such $(i,j)$, that $|Y_{i,j}|<|O_{i,j}|$ and that both $(Y_{i,j})$ and $(O_{i,j})$ are minimal primes of $I_{i,j}$.

Suppose that $(r,n)\leq (i,j) <(i',j')\leq (n,n)$ and that $(i',j')$ covers $(i,j)$ in the order $<$, i.e., either $j<n$ and $(i',j') = (i,j+1)$ or $j = n$ and $(i',j') = (i+1,1)$.

Define

\begin{equation*}
Y_{i',j'} =
    \begin{cases}
        Y_{i,j} \cup \{z_{i',j'}\} & \text{if } I_{i',j'} \not\subseteq (Y_{i,j}).\\
        Y_{i,j} & \text{if } I_{i',j'} \subseteq (Y_{i,j}).
    \end{cases}
\end{equation*}

and 

\begin{equation*}
O_{i',j'} =
    \begin{cases}
        O_{i,j} \cup \{z_{i',j'}\} & \text{if } I_{i',j'} \not\subseteq (O_{i,j}).\\
        O_{i,j} & \text{if } I_{i',j'} \subseteq (O_{i,j}).
    \end{cases}
\end{equation*}

Note that the minimal generators of $I_{i',j'}$ that are not elements of $I_{i,j}$ are exactly those that are divisible by $z_{i',j'}$.  Thus, $I_{i',j'} \not\subseteq (O_{i,j})$ (respectively, $I_{i',j'} \not\subseteq (Y_{i,j})$) if and only if there exists some minimal generator $\mu$ of $I_{i',j'}$ divisible by $z_{i',j'}$ that is not divisible by any element of $O_{i,j}$ (respectively, of $Y_{i,j}$).  Hence, arguing by induction, it follows from Lemma \ref{lem:findminprimes} that both $(O_{i',j'})$ and $(Y_{i',j'})$ are minimal primes of $I_{i',j'}$

We now claim that, for all $(i',j')\geq (r,n)$, $z_{i',j'} \in Y_{i',j'}$ if and only if $z_{i',j'} \in O_{i',j'}$. From this claim it will be immediate that $|Y_{i',j'}|>|O_{i',j'}|$ and, in particular, that $|Y_{n,n}|>|O_{n,n}|$.

Noting that the claim is true for $(r,n)$ itself, we fix some $(i,j)$ satisfying $(r,n) \leq (i,j) < (n,n)$ and assume the truth of the claim for all $(a,b)$ satisfying $(r,n) \leq (a,b) \leq (i,j)$. It then suffices to prove the claim for $(i',j')$ covering $(i,j)$.

Suppose that $z_{i',j'} \in Y_{i',j'}$, equivalently $I_{i',j'} \not\subseteq (Y_{i,j})$. Then there exists some generator $\mu$ of $I_{i',j'}$ so that $z_{i',j'} \mid \mu$ and $\mu \notin (Y_{i,j})$. If $\mu$ is the initial term of a Fulton generator determined by an essential cell whose column index is $<c$, then $\deg(\mu) = 0$, and so $\mu \in \Dom(A)$, a contradiction.  Hence, by property (4) of the hypotheses, $\mu$ is the initial term of a Fulton generator determined by an essential cell whose column index is $>c$. Call that essential cell $(r',c')$. 

Let $\nu$ be the product of the variables dividing $\mu$ whose row index is at most $r$.  We consider two cases: $\deg(\nu) = r$ and $\deg(\nu)<r$.

First suppose that $\deg(\nu) = r$.  By property (2) of the hypotheses, we know that $\nu$ is divisible only by variables in columns with index at least $c$.  By construction of $Y_{r,n}$ and $O_{r,n}$, an antidiagonal of length $r$ in the first $r$ rows of $Z$ and columns weakly east of $c$ that is disjoint from $\Dom(A)$ either contains both an element of $Y_{r,n}$ and an element of $O_{r,n}$ or neither. Thus, $\mu \notin (Y_{i,j})$ implies $\nu \notin (Y_{r,n})$ implies $\nu \notin (O_{r,n})$.  Because $Y_{i,j}$ and $O_{i,j}$ agree below row $r$ by induction, it follows that $\mu \notin (O_{i,j})$.

Alternatively, suppose $\deg(\nu)<r$. Let $Z'$ be the $\deg(\nu) \times \deg(\nu)$ submatrix of $Z$ with row indices $\{r-\deg(\nu)+1, \ldots, r\}$ and column indices $\{c'-\deg(\nu)+1, \ldots, c'\}$. Then set $\mu' = (z_{r,c'-\deg(\nu)+1}z_{r-1,c'-\deg(\nu)+2} \cdots z_{r-\deg(\nu)+1,c'})\mu/\nu$, and note that $\mu' \in I_{i',j'}$ because it is the product of the antidiagonal entries of $Z'$ with the antidiagonal entries of the submatrix of $Z$ below row $r$ giving rise to $\nu$. That is, it is the product of antidiagonal entries of a $\deg(\mu) \times \deg(\mu)$ submatrix of $Z$ weakly northwest of $(r',c')$ and so the lead term of an element of $F_{i',j'}$.  Said otherwise, we are replacing the submatrix giving rise to $\mu$ with a submatrix whose first $\deg(\nu)$ rows and final $\deg(\nu)$ columns are maximally southeast subject to the condition of remaining weakly north of row $r$ and weakly west of the essential box giving rise to $\mu$. See \Cref{fig:shiftsoutheast} for an illustration of the replacement of $\mu$ by $\mu'$.

\begin{figure}[htp]
\[
A = \begin{pmatrix}
    0 & 0 & 0 & 1 & 0 & {\color{orange}\boxed{\color{black}0}} & 0 & 0\\
    0 & 0 & {\color{orange}\boxed{\color{black}1}} & 0 & 0 & 0 & 0 & 0\\
    0 & 0 & {\color{orange}\boxed{\color{black}0}} & 0 & 1 & 0 & 0 & 0\\
    0 & {\color{yellow}\boxed{\color{black}1}} & \boxed{{\color{yellow}\boxed{\color{black}-1}}} & 0 & 0 & \boxed{0} & 1 & 0\\
    1 & 0 & 0 & 0 & 0 & 0 & 0 & 0\\
    0 & 0 & 1 & 0 & 0 & 0 & 0 & 0\\
    0 & 0 & 0 & 0 & 0 & \boxed{0} & 0 & 1\\
    0 & 0 & 0 & 0 & 0 & 1 & 0 & 0
\end{pmatrix}
\]
    \caption{In this example, $r = 4$, $c=2$, and $n=8$. Notice that $\ess(A) \setminus \Dom(A) = \{(4,3), (4,6), (7,6)\}$, whose locations are indicated with a black box.  Locations of elements of $Y_{6,8}$ are indicated with a yellow box, and locations of elements of $O_{6,8}$ are indicated with an orange box. Consider $(r',c') = (7,6)$, and note $\rk_A(7,6) = 5$.  Then $\mu = (z_{1,6}z_{2,5}z_{3,4})(z_{5,3}z_{6,2}z_{7,1})$ witnesses $I_{7,1} \not\subseteq (Y_{6,8})$ but does not witness $I_{7,1} \not\subseteq (O_{6,8})$.  The modification $\mu' = (z_{2,6}z_{3,5}z_{4,4})(z_{5,3}z_{6,2}z_{7,1})$ witnesses $I_{7,1} \not\subseteq (O_{6,8})$.}
    \label{fig:shiftsoutheast}
\end{figure}

Because $\deg(\nu)<r$, $\mu'$ is not divisible by any element in row $1$ or any element in column $c+1$ except possibly $z_{r,c+1}$.  Hence, using that $O_{i,j}$ and $Y_{i,j}$ agree by induction below row $r$, $\mu' \notin (O_{i,j})$.  

Thus, in all cases, $I_{i',j'} \not\subseteq (O_{i,j})$, and so $z_{i',j'} \in O_{i',j'}$, as desired. 

The proof that $z_{i',j'} \in O_{i',j'}$ implies $z_{i',j'} \in Y_{i',j'}$ is analogous. We conclude by induction that $|O_{n,n}|<|Y_{n,n}|$ and, having shown that both $(O_{n,n})$ and $(Y_{n,n})$ are minimal primes for $I_{n,n}$, that $I_{n,n}$ and so also $\init I_A$ are not height unmixed, which implies that $A$ is not equidimensional. \end{proof}

\section*{Acknowledgements}
This project began as part of the University of Minnesota School of Mathematics Summer 2022 REU program. The authors would like to thank John O'Brien for his assistance as a TA during that program as well as all of the REU organizers, mentors, and participants for valuable feedback throughout the summer.  They particularly thank Sterling Saint Rain for very regular and engaged participation.

Additionally they would like to thank Mike Cummings and  Adam Van Tuyl for sharing an early version of their new Macaulay2 package GeometricDecomposability, now described in \cite{CV24}, and for additional coding advice.  

The authors thank Anna Weigandt for helpful conversations and Frank Sottile for communication concerning \cite{BB03} and \cite{BS98}.

\bibliography{bibliography}{}
\bibliographystyle{alpha}

\end{document}